\documentclass[12pt]{article}
\usepackage[mathscr]{eucal}
\usepackage{amssymb}
\usepackage{latexsym}
\usepackage{amsthm}
\usepackage{amsmath}
\usepackage[pdftex]{graphicx} 
\usepackage{psfrag}
\usepackage{a4wide}

\newcommand{\R}{\mathbb{R}}

\newcommand{\A}{\boldsymbol{A}}

\newcommand{\E}{\boldsymbol{E}}
\newcommand{\J}{\boldsymbol{J}}
\newcommand{\I}{\boldsymbol{I}}
\newcommand{\D}{\boldsymbol{D}}
\newcommand{\F}{\boldsymbol{F}}
\newcommand{\M}{\boldsymbol{M}}
\newcommand{\Om}{\boldsymbol{O}}
\newcommand{\Pm}{\boldsymbol{P}}

\newcommand{\X}{\boldsymbol{X}}
\newcommand{\x}{\boldsymbol{x}}
\newcommand{\y}{\boldsymbol{y}}
\newcommand{\av}{\mathfrak{a}}
\newcommand{\vf}{\mathfrak{v}}
\newcommand{\ev}{\boldsymbol{e}}
\newcommand{\jv}{\boldsymbol{j}}
\newcommand{\vv}{\boldsymbol{v}}
\newcommand{\uv}{\boldsymbol{u}}

\newcommand{\el}{\boldsymbol{\ell}}

\newcommand{\la}{\langle}
\newcommand{\ra}{\rangle}

\newtheorem{theorem}{Theorem}[section]

\newtheorem{lemma}[theorem]{Lemma}

\theoremstyle{definition}

\newtheorem{remark}[theorem]{Remark}

\makeatletter
\@addtoreset{equation}{section}

\makeatother

\title{Embedding dimensions of matrices whose entries are indefinite distances in the pseudo-Euclidean space
}
\author{
Hiroshi Nozaki\thanks{Department of Mathematics Education, 
	Aichi University of Education, 
	1 Hirosawa, Igaya-cho, 
	Kariya, Aichi 448-8542, 
	Japan. {\tt hnozaki@auecc.aichi-edu.ac.jp}}
\and 
Masashi Shinohara\thanks{Faculty of Education, 
	Shiga University,
	2-5-1 Hiratsu, Otsu, Shiga 520-0862, 
	Japan.  
	{\tt shino@edu.shiga-u.ac.jp}}
\and
Sho Suda\thanks{Department of Mathematics, National Defense Academy of Japan, Yokosuka, Kanagawa 239-8686, Japan.
{\tt ssuda@nda.ac.jp}}
}

\begin{document}
\maketitle

\abstract{
A finite set of the Euclidean space is called an $s$-distance set provided the number of Euclidean distances in the set is $s$. Determining the largest possible $s$-distance set for the Euclidean space of a given dimension is challenging. This problem was solved only when dealing with small values of $s$ and dimensions. Lison\v{e}k (1997) achieved 
the classification of the largest 2-distance sets for dimensions up to $7$, using computer assistance and graph representation theory. 
In this study, we consider a theory analogous to these results of Lison\v{e}k for the pseudo-Euclidean space $\mathbb{R}^{p,q}$. We consider an $s$-indefinite-distance set in a pseudo-Euclidean space that uses the value
\[
|| \x-\y ||=(x_1-y_1)^2 +\cdots +(x_p -y_p)^2-(x_{p+1}-y_{p+1})^2-\cdots -(x_{p+q}-y_{p+q})^2
\] instead of the Euclidean distance. 
We develop a representation theory for symmetric matrices in the context of $s$-indefinite-distance sets, which includes or improves the results of Euclidean $s$-distance sets with large $s$ values. Moreover, we classify the largest possible $2$-indefinite-distance sets for small dimensions. 
 }
\vspace{0.5cm}


\section{Introduction}
A finite subset $X$ of the Euclidean space $\mathbb{R}^d$ is called an {\it $s$-distance set} if the number of the distances between two distinct points in $X$ is $s$. 
The cardinality of an $s$-distance set in $\mathbb{R}^d$ or the $(d-1)$-dimensional sphere, denoted as $S^{d-1}$, has a natural upper bound, and determining the largest possible $s$-distance set for given $d$ and $s$ values is a major challenge. 
Several results on the largest $s$-distance sets have been obtained \cite{DGS77, GY18, MN11, NS21, Spre, SO20}. The largest $s$-distance sets are closely related to good combinatorial structures like designs or codes \cite{BB09, BS81,CK07,D73,DGS77,DL98}. 

For the classification of the largest $s$-distance sets for small dimensions, 
 considering the representations of graphs with $s$ relations is useful. 
For disjoint symmetric binary relations $R_0=\{(x,x) \mid x \in S\},R_1,\ldots,R_s$ on a finite set $S$ with $S\times S=\bigcup_{i=0}^s R_i$, we consider a map $f$ from $S$ to $\mathbb{R}^d$ satisfying that there exist positive real numbers 
$a_1,\ldots ,a_s$ such that the distance of $f(x)$ and $f(y)$ is $a_i$ if $(x,y) \in R_i$. 
The image $f(S)$ is  a {\it representation of $(S,\{R_i\}_{i=0}^s)$} as an $s$-distance set in $\mathbb{R}^d$. 
Finding representations in minimal dimensions for a given size of $S$ and relations is challenging in general. 
For $2$-distance sets, only $2$ minimal-dimensional representations exist for a given graph \cite{ES66}, and Lison\v{e}k \cite{L97} classified the largest 2-distance sets in $\mathbb{R}^d$ for $d\leq 7$ using computer support. 
Moreover, for $2$-distance sets, the minimal dimensions of representations are explicitly obtained from the spectral information of the adjacency matrices of a graph \cite{R10}. 
Moreover, it is possible to determine whether the representation is on a sphere \cite{NS12}. In this study, we consider extensions of such results for the pseudo-Euclidean space $\R^{p,q}$.  
The central problem addressed in this study was investigated in previous works \cite{BT,L97}.

Let $\mathbb{R}^{p,q}$ be the $(p+q)$-dimensional linear space over $\mathbb{R}$ with the bilinear form 
\[
\langle \langle \x, \y \rangle \rangle=x_1 y_1 +\cdots +x_p y_p-x_{p+1}y_{p+1}-\cdots -x_{p+q}y_{p+q}
\] 
for $\x=(x_1,\ldots,x_{p+q})^\top, \y=(y_1,\ldots,y_{p+q})^\top \in \mathbb{R}^{p+q}$. 
For a finite subset $X$ of $\R^{p,q}$, we define
\[
A(X)=\{|| \x-\y||\colon\, \x, \y  \in X, \x \ne \y \},
\]
where $||\x||=\la\la\x,\x \ra\ra $. 
A finite subset $X$ of $\R^{p,q}$ is called an {\it $s$-indefinite-distance set} if  $|A(X)|=s$. It is noteworthy that $||\x-\y||$ is not a distance function except for the Euclidean case $\R^p=\R^{p,0}$, and it may take on negative values. The $s$-indefinite-distance set in $\R^{p,0}$ is the $s$-distance set.  
Two subsets $X$ and $Y$ of $\R^{p,q}$ are {\it isomorphic} as $s$-indefinite-distance
sets if one can be transformed to the other using a function defined by $f(\boldsymbol{x})=R \boldsymbol{x}+\boldsymbol{b}$, where $R$ is an element of the pseudo-orthogonal group $ O(p,q)$ and $\boldsymbol{b}\in \R^{p,q}$.
The $s$-indefinite-distance sets in $\R^{p,q}$ are generally considered up to isometry. An $s$-indefinite-distance set in $\R^{p,q}$ is said to be {\it proper} if it is not in $\R^{k,l}$ for $(k,l)\ne (p,q)$ with $k \leq p$ and $l \leq q$. 
Bannai, Bannai, and Stanton \cite{BBS83} gave an upper bound $|X| \leq  \binom{p+s}{s}$ for a Euclidean $s$-distance set $X$ in $\R^{p,0}$. Petrov and Pohoata \cite{PP21} also proved this bound using the Croot--Lev--Pach lemma \cite{CLP17}. 
The upper bound for $s$-indefinite-distance sets in $\R^{p,q}$ is analogously obtained using the method given in \cite{PP21} provided $0 \not\in A(X)$. 
We can construct infinite $s$-indefinite-distance sets with $0 \in A(X)$ for all dimensions, except for $(p,q)= (1,1)$. An $s$-indefinite-distance set in $\R^{1,1}$ with $0\in A(X)$ is finite. 
The problem we address in this study is determining the largest possible $s$-indefinite-distance sets for $0 \not \in A(X)$. 

For this purpose, 
we would like to give the minimal-dimensions of  representations of $(S,\{R_i\}_{i=0}^s)$ in $\R^{p,q}$ using the spectral information of the relation matrices. 
For the given relations $\{R_i\}_{i=0}^s$ on $S$, 
a relation matrix $\A_i$ with respect to $R_i$ is defined as a symmetric matrix. 
The rows and columns of this matrix are indexed by $S$, and the entries in the matrix are $(\A_i)_{uv}=1$ when $(u,v)\in R_i$, and 0 otherwise. A matrix expressed by $\M=\sum_{i=0}^s c_i \A_i$ with $c_i \in \mathbb{R}$ is called a {\it dissimilarity matrix} of $(S,\{R_i\})$. 
Let $\Pm=\I-(1/|S|) \J$, where $\I$ is the identity matrix and $\J$ is the all-one matrix. When the signature (numbers of positive and negative eigenvalues) of $-\Pm \M \Pm$ is $(p,q)$, it indicates the presence of a subset $X$ of $\R^{p,q}$ whose indefinite-distance matrix $(||\x-\y||)_{\x,\y \in X}$ is $\M$ \cite{G85}. The set $X \subset \R^{p,q}$ is called a {\it representation of dissimilarity matrix $\M$}. 
Moreover, there is no subset $X$ of $\R^{p_1,q_1}$ fulfilling $(||\x-\y||)_{\x,\y \in X} =\M$ for either $p_1 < p$ or $q_1 < q$~\cite{G85}.
Roy \cite{R10} determined the exact minimal dimension $p$ from the spectral information of $\M$ for $s=2$ and $q=0$. We generalize his result to the case of any $s$, $p$, and $q$ (Theorem~\ref{thm:dim_PMP}). This result can be used to determine when the representation is on a sphere $S_{p,q}(r)=\{\x \in \mathbb{R}^{p,q} \mid \la \la \x,\x \ra \ra =r\}$. Moreover, from these results, we can extend the Lison\v{e}k algorithm to classify the largest proper (spherical)  $2$-indefinite-distance sets in $\mathbb{R}^{p,q}$ for $p+q \leq 7$. 

The remainder of this paper is organized as follows. In Section \ref{sec:2}, we prove upper bounds on the cardinality of an $s$-indefinite-distance set in $\R^{p,q}$ with $0 \not \in A(X)$, and consider the case $0 \in A(X)$. In Section \ref{sec:3}, we consider the minimal dimensions $p,q$, for which a representation of a given dissimilarity matrix $\M=\sum c_i \A_i$ is feasible. In Section \ref{sec:4}, we classify the largest proper $2$-indefinite-distance sets in $\R^{p,q}$ for $p+q\leq 7$ using an extension of the Lison\v{e}k algorithm. In Section \ref{sec:5}, we provide a necessary and sufficient condition for a given indefinite-distance matrix $\M$ to be spherical. 
Moreover, we classify the largest proper spherical $2$-indefinite-distance sets in $\R^{p,q}$ for $p+q\leq 7$. 

\section{Upper bounds for $s$-indefinite-distance sets in $\mathbb{R}^{p,q}$} \label{sec:2}

Petrov and Pohoata \cite{PP21} provided a simple proof of the upper bounds $|X| \leq \binom{d+s}{s}$ for Euclidean $s$-distance sets in $\mathbb{R}^d$ using the Croot--Lev--Pach lemma~\cite{CLP17}. 
The upper bound for $s$-indefinite-distance sets in $\mathbb{R}^{p,q}$ can be obtained using the method described in \cite{PP21}.  
Theorem~\ref{thm:key} is a key theorem presented in \cite{PP21}. 
The pair $(r,s)$ is called the {\it signature}  of symmetric matrix $\M$, where $r$ ({\it resp}.\ $s$) is the number of positive ({\it resp}.\ negative) eigenvalues of $\M$ counting multiplicities.  
Let ${\rm sign}(\M)$ denote the signature of $\M$. 
\begin{theorem}[{\cite{PP21}}]\label{thm:key}
Let $V$ be a finite-dimensional vector space over a field $\mathbb{F}$, and $X$ be a finite set in $V$. Let $s$ be a nonnegative integer. Let $F(\x,\y)$ be a $2\cdot \dim V$-variate polynomial with coefficients in $\mathbb{F}$ and a degree no greater than $2s+1$. 
Let the matrix $\M=(F(\x,\y))_{\x,\y \in X}$, and $(r_+,r_-)={\rm sign} (\M)$. Let $\dim_s(X)$ be the dimension of the space of polynomials of degree at most $s$ considered as functions on $X$. Then the following hold. 
\begin{enumerate}
\item ${\rm rank}(\M) \leq 2 \dim_s(X)$.   
\item If $\mathbb{F}=\mathbb{R}$, then $\max\{r_+,r_-\} \leq \dim_s(X)$. 
\end{enumerate} 
\end{theorem}
\begin{theorem}\label{thm:absolute_bound}
Let $X$ be an $s$-indefinite-distance set in $\mathbb{R}^{p,q}$. 
If $0 \not\in A(X)$ holds, then 
\[
|X| \leq \binom{p+q+s}{s}. 
\]
\end{theorem}
\begin{proof}
Define the polynomial $F$ of degree $2s$ as   
\[
F(\x,\y)=\prod_{\alpha \in A(X)}(\alpha-||\x-\y||)/\alpha  
\]
for $0 \not\in A(X)$. 
This polynomial satisfies $F(\x,\x)=1$ and $F(\x,\y)=0$ for distinct $\x , \y \in X $. 
Here, $\M=(F(\x,\y))_{\x,\y \in X}$ is the identity matrix. Therefore,  Theorem~\ref{thm:key} implies that
\[
|X| = r_+ \leq \dim_s(X) \leq \dim_s(\mathbb{R}^{p+q})=\binom{p+q+s}{s}. \qedhere
\]
\end{proof}

If $0 \in A(X)$ and $(p,q)\ne (1,1)$, then 
there exists a proper $s$-indefinite-distance set $X$ in $\mathbb{R}^{p,q}$ with $|X|=\infty$ . This can be expressed as follows.
We can assume that $p\geq 2$ and $q\geq 1$ because the constructions are similar for $p=1$ and $q\geq 2$. 
Let $L_s$ be a Euclidean $s$-distance set in $\mathbb{R}^1$. For instance, $L_s$ may be 
a set of $s + 1$ points, wherein two consecutive points maintain an equal interval. 
We define $L_0$ to be an empty set.  Since $L_s$ is Euclidean, $0 \not\in A(L_s)$. 
Let $V={\rm Span}_\mathbb{R}\{\boldsymbol{e}_2+\boldsymbol{e}_{p+1}\}$, where $\boldsymbol{e}_i$ is the vector with $i$-th entry 1, and 0 otherwise. 
It is noteworthy that $A(V)=\{0\}$. Let $L_{s}'=\{(\ell,0,\ldots,0) \in \mathbb{R}^{p,q} \mid \ell \in L_{s} \}$ and \[
X=\{ \el+\vv \in \mathbb{R}^{p,q} \mid \el \in  L_{s-1}' , \vv \in V\}.
\] 
Therefore, for any $\x=\el+\vv, \x'=\el'+\vv' \in X$, 
\[
||\x-\x'||=||\el-\el'||+||\vv-\vv'|| \in A(L_{s-1}) \cup \{0\},
\]
and $|A(X)|=s$. This set $X$ is an infinite $s$-indefinite-distance set in $\mathbb{R}^{p,q}$ with distance $0$. 

The set $\{(x,y) \in \mathbb{R}^{1,1} \mid x=y \}$ is an infinite $1$-indefinite-distance set in $\R^{1,1}$ with $A(X)=\{0\}$.  
For $s\geq 2$, the cardinality of an $s$-indefinite-distance set in $\mathbb{R}^{1,1}$ is finite even if $0 \in A(X)$. 
This is substantiated below.  
For $\x=(x_1,x_2) \in \mathbb{R}^{1,1}$ and  $a\in \mathbb{R}$, we define 
\[D_{\x}(a)=\{\x'\in \mathbb{R}^{1,1}\colon\, ||\x-\x'||=a\}.\]  
If $a=0$, then $D_{\x}(a)$ is the union of two lines $\ell ^{\pm}$ with slopes $\pm 1$ and intersecting at $(x_1,x_2)$. 
If $a\ne 0$, $D_{\x}(a)$ is hyperbolic and its asymptotes are $\ell ^{\pm}$.
It should be noted that $D_{\x}(a)\cap D_{\x'}(a')$ is infinite 
for distinct $\x, \x'\in \mathbb{R}^{1,1}$ 
if and only if $a=a'=0$ and $||\x-\x'||=0$. 
Suppose any $\x, \x'\in X$ satisfy $||\x-\x'||\ne 0$. 
Then $|X|$ is finite since 
$$X\setminus \{\x, \x'\}\subset 
\bigcup _{a, a'\in A(X)} (D_{\x}(a) \cap D_{\x'}(a')).$$
Therefore, an $s$-indefinite-distance set $X\subset \mathbb{R}^{1,1}$ for $s\ge 2$
is finite even for $0\in A(X)$.

\section{Embedding dimensions of dissimilarity matrices} \label{sec:3}
Let $V$ be a finite set, with $|V|=n$. 
We call $\mathscr{R}=\{R_0,R_1,\ldots, R_s\}$ {\it  relations} on $V$ if $\mathscr{R}$ is a partition of $V\times V$, $R_0=\{(u,u) \mid u \in V \}$, and $R_i$ is a nonempty symmetric relation  for each $i$. 
In this paper, we call $(V,\mathscr{R})$ an {\it $s$-relation graph}. 
A matrix $\A_i$ is a {\it relation matrix} with respect to $R_i$ if $\A_i$ is a symmetric $(0,1)$-matrix indexed by $V$ with $(u,v)$-entries
\[
\A_i(u,v)=\begin{cases}
1 \text{ if $(u,v) \in R_i$},\\
0 \text{ otherwise}.
\end{cases}
\]
 A matrix $\D_{\mathscr{R}}(\av)= \D(\av)=\sum_{i=1}^s a_i \A_i$ is called a {\it dissimilarity matrix on $(V,\mathscr{R})$} for  $\av=(a_1,\ldots, a_s)^\top \in \R^s$.  

For $X \subset \R^{p,q}$, the matrix $\D(X)=(||\x-\y||)_{\x,\y \in X}$ is called an {\it indefinite-distance matrix of $X$}. 
A dissimilarity matrix $\D_{\mathscr{R}}(\av)$ on $(V,\mathscr{R})$ is {\it representable} in $\R^{p,q}$ if there exists $X\subset \R^{p,q}$ such that  $\D_{\mathscr{R}}(\av)=\D(X)$. This set $X$ is called a {\it representation of $\D_{\mathscr{R}}(\av)$}. 
An $s$-relation graph $(V,\mathscr{R})$ is {\it representable in $\R^{p,q}$} if there exists $\av \in \mathbb{R}^s$ such that $\D_{\mathscr{R}}(\av)=\D(X)$ for some $X \subset \R^{p,q}$. 
Moreover, the set $X$ is called a {\it representation of $(V,\mathscr{R})$ in $\R^{p,q}$}. 
Let $\X$ be an $n\times (p+q)$ matrix whose rows
 are the coordinates of the $n$ points of $X\subset \R^{p,q}$. 
The {\it dimensionality} of $\D_{\mathscr{R}}(\av)$ is the least ${\rm rank}(\X)$ in all representations $X$ of  $\D_{\mathscr{R}}(\av)$. 

Let $\jv$ be the all-one vector. For a symmetric matrix $\M$ and a vector $\el$ with $\el^\top \jv=1$, let $\F_{\M}(\el)$ be the symmetric matrix defined as 
\[
\F_{\M}(\el)=- (\I-\jv \el^{\top})\M (\I-\el \jv^{\top}),
\]
where $\I$ is the identity matrix. 
Note that the matrix $\I-\el \jv^{\top}$ is a projection matrix onto 
$\jv^\perp$. From $(\I-\el \jv^{\top})\el=0$, the vector $\el$ is an eigenvector of $\F_{\M}(\el)$ with an eigenvalue of 0.
Direct calculations show that \[
\F_{\D(X)}(\el)=
(-2\la\la \x-\uv,\y-\uv\ra \ra )_{\x,\y \in X},\]
 where $\uv=
\sum_{\x \in X} \el_{\x}\x$ and $\el_{\x}$ is the $\x$-th entry of $\el$. 
From the diagonalization of a symmetric matrix $\M$ with signature $(r,s)$,  we obtain the coordinates of the points of $X\subset \R^{p,q}$ such that $\M=(\la\la \x,\y \ra\ra)_{\x,\y \in X}$.  
Gower \cite{G85} showed that the signature of 
$\F_{\D_\mathscr{R}(\av)}(\el)$ is the same for all 
$\el$ with $\el^\top \jv=1$ and obtained the following theorem: 
\begin{theorem}[Theorems 8 and 9 in \cite{G85}] \label{thm:G}
Let $(p,q)$ be the signature of $\F_{\D_\mathscr{R}(\av)}(\el)$, which does not depend on the choice of $\el$.  Then the dimensionality of $\D_\mathscr{R}(\av)$ is $p+q$.  Moreover, 
there is no representation of $\D_\mathscr{R}(\av)$ in $\R^{s,t}$ that satisfies $s< p$ or $t<q$.  
\end{theorem}  
The smallest dimension given in Theorem~\ref{thm:G} as ${\rm sign}(\F_{\D_\mathscr{R}(\av)}(\el))=(p,q)$ is called the {\it embedding dimension} of a dissimilarity matrix $\D_\mathscr{R}(\av)$.  
The signature of $\F_{\D_\mathscr{R}(\av)}(\el)$ can be calculated
using the eigenvalues and main angles of  $\D_\mathscr{R}(\av)$ from Theorem~\ref{thm:dim_PMP}.  
Let $\M$ be a real symmetric matrix of order $n$. Let $\tau_i$ be an eigenvalue of $\M$ with eigenspace $E_i$ and an orthogonal projection matrix $\E_i$ onto $E_i$ for $i=1,\ldots, r$. 
The {\it main angle} of $\tau_i$ (or $E_i$) is defined as 
\[
\beta_i=\frac{1}{\sqrt{n}}\sqrt{ (\E_i \jv)^\top (\E_i \jv)}. 
\]
It is noteworthy that $0\leq \beta_i \leq 1$ and $\sum_{i=1}^r \beta_i^2=1$ hold, and $\beta_i=0$ if and only if $E_i \subset \jv^\perp$. 
An eigenvalue $\tau_i$ is {\it main} if $\beta_i \ne 0$. 
For a main eigenvalue $\tau_i$,  we deduce that
\[
n\beta_i^2=\max_{\vv \in E_i: \vv^\top \vv=1} (\vv^\top \jv)^2=
\max_{\vv \in E_i}\frac{(\vv^\top \jv)^2}{\vv^\top \vv} 
\]
and 
\begin{equation} \label{eq:mainangle}
\frac{1}{n\beta_i^2}= \min_{\vv \in E_i} \frac{\vv^\top \vv}{(\vv^\top \jv)^2}=
\min_{\vv \in E_i:\, \vv^\top \jv=1} \vv^\top \vv.  
\end{equation}

The main theorem of this section is as follows:  
\begin{theorem} \label{thm:dim_PMP}
Let $\M$ be a real symmetric matrix with distinct main eigenvalues
$\lambda_1, \ldots,\lambda_r$.
Let $\beta_i$ be the main angle of $\lambda_i$ for $i \in \{1,\ldots, r\}$. 
Let $(p,q)$ be the signature of $\M$.
Let $E_0$ be the eigenspace associated with eigenvalue $0$ of $\M$, and $E_0=\emptyset$ if $0$ is not an eigenvalue of $\M$.   Then the signature of $\F_{\M}(\el)$ is given as follows. 
\begin{enumerate}
\item If $E_0 \subset \jv^\perp$ and $\sum_{i=1}^r \beta_i^2/\lambda_i=0$, 
then the signature of $\F_{\M}(\el)$ is $(q-1,p-1)$. 
\item If $E_0 \subset \jv^\perp$ and $\sum_{i=1}^r \beta_i^2/\lambda_i>0$, 
then the signature of $\F_{\M}(\el)$ is $(q,p-1)$. 
\item If $E_0 \subset \jv^\perp$ and $\sum_{i=1}^r \beta_i^2/\lambda_i<0$, 
then the signature of $\F_{\M}(\el)$ is $(q-1,p)$. 
\item If $E_0 \not\subset \jv^\perp$, then the signature of $\F_{\M}(\el)$ is $(q,p)$.
\end{enumerate}

\end{theorem} 
\begin{proof}
First, we consider the case $E_0 \subset \jv^\perp$, namely, $0$ is not a main eigenvalue of $\M$.  
We may suppose $\lambda_1<\lambda_2< \cdots <\lambda_r$ and $\lambda_1<0$; 
if needed, we consider $-\M$. 
Let $\E_i$ be the orthogonal projection matrix onto the eigenspace $E_i$ associated with $\lambda_i$. 
Let $\vv_i$ be the vector satisfying $\jv^\top \vv_i=1$ in the same direction as $\E_i\jv$. Note that $\vv_i^\top \vv_i=1/n \beta_i^2$ from \eqref{eq:mainangle}.
 We consider
the signature of $-\F_{\M}(\vv_1)$. 
A non-main eigenvalue of $\M$ is an eigenvalue of $ - \F_{\M}(\vv_1) $, and the corresponding eigenvectors are the same.  
Indeed, for an eigenvector $\uv \in \jv^\perp$ of a non-main eigenvalue $\lambda$ of $\M$, 
\begin{align*}
    -\F_{\M}(\vv_1) \uv &= (\I-\jv \vv_1^\top)\M( \I -\vv_1\jv^\top) \uv\\
        &=(\I-\jv \vv_1^\top)\M\uv\\
        &=\lambda (\I-\jv \vv_1^\top)\uv\\
        &=\lambda \uv, 
\end{align*}
since
$\vv_1^\top\uv=0$ and $\I -\vv_1\jv^\top$ is a projection matrix onto $\jv^\perp$. 
A main eigenvalue of $\M$ is also an eigenvalue of $ - \F_{\M}(\vv_1) $ if we use the eigenvectors in $\jv^\perp \cap E_i$. 
 Clearly, $\vv_1$ is an eigenvector corresponding to eigenvalue 0 of $-\F_{\M}(\vv_1)$.  
Therefore, we must determine the signs of the eigenvalues of $- \F_{\M}(\vv_1)$ 
on ${\rm Span}_{\mathbb{R}}\{\vv_2,\ldots, \vv_r \}$. 

For $\x=\sum_{i=2}^r a_i \vv_i$ with $\jv^\top \x=1$,
 namely $\sum_{i=2}^r a_i=1$, 
one has
\begin{align}
-\F_{\M}(\vv_1)\x
&=(\I-\jv \vv_1^\top) (\sum_{i=2}^r a_i \lambda_i \vv_i-\lambda_1 \vv_1) \nonumber\\
&=\sum_{i=2}^r a_i \lambda_i \vv_i-\lambda_1\vv_1 + \frac{\lambda_1}{n \beta_1^2}\jv \nonumber  \\
&=\sum_{i=2}^r a_i \lambda_i \vv_i-\lambda_1\vv_1 + \frac{\lambda_1}{n \beta_1^2}\sum_{i=1}^rn\beta_i^2 \vv_i \label{eq:j=v_i} \\
&=\sum_{i=2}^r(\lambda_i a_i+\lambda_1 \frac{\beta_i^2}{\beta_1^2})\vv_i.  \nonumber 
\end{align}
In \eqref{eq:j=v_i}, the equality $\jv=\sum_{i=1}^rn\beta_i^2 \vv_i$ is applied. 
This equality is proven by the fact that $\jv$ is a linear combination of $\{\vv_1,\ldots, \vv_r\}$ and $(n\beta_i^2) \jv^\top \vv_i=n\beta_i^2$ is the coefficient of $\vv_i$. 
From the equation $-\F_{\M}(\vv_1)\x=\mu \x$
for variable $\mu$ and the linear independence of $\vv_i$, we obtain
\[
a_i=\frac{\lambda_1}{\beta_1^2}\cdot \frac{\beta_i^2}{\mu-\lambda_i}
\] 
for $i\in \{2,\ldots,r\}$. 
By $\sum_{i=2}^r a_i=1$, we have
\begin{equation} \label{eq:mu}
\frac{\lambda_1}{\beta_1^2}\sum_{i=2}^r \frac{\beta_i^2}{\mu-\lambda_i}=1. 
\end{equation}
Multiplying both sides of  
\eqref{eq:mu} by $f(\mu)=\prod_{i=2}^r (\mu-\lambda_i)$ , 
\[
\frac{\lambda_1}{\beta_1^2}\sum_{i=2}^r \frac{\beta_i^2}{\mu-\lambda_i}f(\mu)=f(\mu). 
\]
Moreover, the following polynomial equation for $\mu$ of degree $r-1$ is obtained:  
\begin{equation} \label{eq:mu2}
g(\mu)=\frac{\beta_1^2}{\lambda_1}f(\mu)-\sum_{i=2}^r \frac{\beta_i^2}{\mu-\lambda_i}f(\mu)=0. 
\end{equation}
The solutions of \eqref{eq:mu2} are the eigenvalues of $-\F_{\M}(\vv_1)$ on ${\rm Span}_{\mathbb{R}}\{\vv_2,\ldots, \vv_r \}$. 

The signs of the solutions were also determined.  
Since we have 
\begin{align*}
    g&(\lambda_i)g(\lambda_{i+1})
    =\left(-\beta_i^2\prod_{j\ne 1,i} (\lambda_i-\lambda_j)\right)\left(-\beta_{i+1}^2\prod_{j\ne 1,i+1} (\lambda_{i+1}-\lambda_j)\right)\\
    &=- \beta_i^2 \beta_{i+1}^2 \left(\prod_{1<j<i}(\lambda_i-\lambda_j)(\lambda_{i+1}-\lambda_j) \right) (\lambda_i-\lambda_{i+1})^2 \left(\prod_{j>i+1}(\lambda_i-\lambda_j)(\lambda_{i+1}-\lambda_j) \right) \\
    &<0
\end{align*}
for each $i \in \{2,\ldots, r-1\}$, 
there exists a solution of $g(\mu)$ between $\lambda_i$ and $\lambda_{i+1}$. 
The remaining solution is smaller than $\lambda_2$ because the leading coefficient ${\beta_1^2}/{\lambda_1}<0$ and $\lim_{t\rightarrow -\infty} g(t)g(\lambda_2)<0$. 
Therefore, the roots $\mu_1,\ldots,\mu_{r-1}$ of $g(\mu)$ are simple and it follows that 
\[
\mu_1<\lambda_2<\mu_2<\lambda_3<\cdots<\lambda_{r-1}<\mu_{r-1}<\lambda_r. 
\]

Let $\lambda_{k}$ be the smallest positive value of $\{\lambda_2,\ldots, \lambda_r\}$. 
The sign of $\mu_{k-1}$ determines the signature of $-\F_{\M}(\vv_1)$. Thus, we consider the sign of $g(0)g(\lambda_{k})$.  
We can calculate
 \[
g(\lambda_{k})g(0)=g(\lambda_k)f(0)\sum_{i=1}^r \frac{\beta_i^2}{\lambda_i} 
\]
and $g(\lambda_k)f(0)>0$. 
If $\sum_{i=1}^r {\beta_i^2}/{\lambda_i}=0$, then 
the signature of $-\F_{\M}(\vv_1)$ is $(p-1,q-1)$.  
If $\sum_{i=1}^r {\beta_i^2}/{\lambda_i}>0$, then 
the signature of $-\F_{\M}(\vv_1)$ is $(p-1,q)$. 
If $\sum_{i=1}^r {\beta_i^2}/{\lambda_i}<0$, then 
the signature of $-\F_{\M}(\vv_1)$ is $(p,q-1)$. 
Therefore, the signature of  $\F_{\M}(\vv_1)$ is  the assertion value for $E_0 \subset \jv^\perp$. 

Suppose $E_0 \not\subset \jv^\perp$, namely, $0$ is a main eigenvalue of $\M$. 
Similar to the case $E_0 \subset \jv^\perp$, we  obtain the polynomial $g(\mu)$ and its roots $\mu_1,\ldots, \mu_{r-1}$. 
Since 0 is a main eigenvalue, $\mu_{i}$ and $\lambda_{i+1}$ have the same sign for $i \in \{1,\ldots, r-1\}$ except for $\lambda_k=0$. The sign of $\mu_{k-1}$ is negative.  
Therefore, the signature of $-\F_{\M}(\vv_1)$ is $(p,q)$, and the assertion holds for $E_0 \not\subset \jv^\perp$. 
\end{proof}

For a given $(V,\mathscr{R})$ with $|V|=n$, we  consider choosing $\av$ such that  the dimensionality of $\D_\mathscr{R}(\av)$ is minimal. 
For $\x \in \mathbb{R}^n$, we define the map $\mathfrak{v}: \mathbb{R}^n \rightarrow \mathbb{R}^s$ as follows: 
 \[
\mathfrak{v}(\x)=-(\x^\top \A_1\x, \x^\top \A_2 \x, \ldots, \x^\top \A_s \x)^\top. 
\]
This vector satisfies that
\[
\mathfrak{v}(\x)^{\top} \av=-\x^\top \D_{\mathscr{R}}(\av)\x.
\]
Let $\mathfrak{V}_\mathscr{R}$ denote the following set, 
\[
\mathfrak{V}_\mathscr{R}=\{\vf(\x) \in \mathbb{R}^s \mid \x \in \jv^{\perp} \subset \mathbb{R}^n, \x^{\top} \x=1 \}.
\] 
For $\jv \in \R^s$ and each $\vf(\x) \in \mathfrak{V}_\mathscr{R}$, 
it follows that 
\begin{align*}
    \jv^\top \vf(\x)=-\x^\top (\J- \I) \x = \x^\top \x=1.  
\end{align*}
because  $\x \in \jv^{\perp}$. Thus $\mathfrak{V}_\mathscr{R}$ is on the hyperplane $\jv^\perp\subset \R^s$. 

For $\el=\hat{\jv}:=(1/n)\jv$, we consider the signature of $\F_{\D_{\mathscr{R}}(\mathfrak{a})}(\hat{\jv})$ to determine the dimensionality of $\D_{\mathscr{R}}(\mathfrak{a})$. 
For $\Pm=\I-(1/n)\J=\I-\jv \el^\top=\I-\el \jv^\top$ and $\mathfrak{a}=(a_1,\ldots,a_s)^\top$, one has
\[
\F_{\D_{\mathscr{R}}(\mathfrak{a})}(\hat{\jv})=
-\Pm \D_{\mathscr{R}}(\mathfrak{a}) \Pm 
=- \sum_{i=1}^s a_i \Pm \A_i \Pm. 
\]
It is difficult to determine the eigenspace of $\F_{\D_{\mathscr{R}}(\mathfrak{a})}(\hat{\jv})$
based on the eigenspaces of $\Pm \A_i \Pm$. 
Thus, we suppose $s$-relation graphs $(V,\mathscr{R})$ 
satisfy that $\Pm\A_1\Pm,\ldots, \Pm\A_s\Pm$ are mutually commutative. If symmetric matrices are mutually commutative, then they have common eigenspaces. 
The class of such $s$-relation graphs $(V,\mathscr{R})$ contains symmetric association schemes \cite{BIb} and the minimal dimensionality of the representations of $(V,\mathscr{R})$ can be verified from the common eigenspace of $\Pm\A_i\Pm$. 
\begin{theorem}\label{thm:dimrep}
Let $\A_1,\ldots, \A_s$ be the relation matrices of an $s$-relation graph  $(V,\mathscr{R})$. 
Suppose $\Pm\A_1\Pm,\ldots, \Pm\A_s\Pm$ are mutually commutative. Let $E_1,\ldots, E_r$ be the common eigenspaces of $\Pm\A_i\Pm$. 
Let $\x_i \in E_i$ and $m_i$ the dimension of $E_i$. Let $\mathfrak{a} \in \mathbb{R}^s$ and define the two sets
\[
\mathfrak{V}_{\mathfrak{a}}^+=\{ \mathfrak{v} \in \mathfrak{V}_{\mathscr{R}} \mid \mathfrak{a}^\top \mathfrak{v}>0\},\qquad 
\mathfrak{V}_{\mathfrak{a}}^-=\{ \mathfrak{v} \in \mathfrak{V}_{\mathscr{R}} \mid \mathfrak{a}^\top \mathfrak{v}<0\}.\]
Then $\D_{\mathscr{R}}(\mathfrak{a})$ is representable in $\mathbb{R}^{p,q}$, where  
\[
p=\sum_{i:\, \vf (\x_i) \in \mathfrak{V}_{\mathfrak{a}}^+}m_i, \qquad 
q=\sum_{i:\, \vf (\x_i) \in \mathfrak{V}_{\mathfrak{a}}^-}m_i.
\]
Moreover, the dimensionality of $\D_{\mathscr{R}}(\mathfrak{a})$ is $p+q$. 
\end{theorem}
\begin{proof}
Let $\lambda_{ij}$ be the eigenvalue of $\Pm \A_i \Pm$ on $E_j$ and $\E_j$ be the orthogonal projection matrix onto $E_j$.
Then, we have
\[
\F_{\D_{\mathscr{R}}(\mathfrak{a})}(\hat{\jv})
=- \sum_{i=1}^s a_i \Pm \A_i \Pm=- \sum_{i=1}^s  a_i  \sum_{j=1}^r\lambda_{ij}\E_j=- \sum_{j=1}^r\left(\sum_{i=1}^s  a_i  \lambda_{ij}\right)\E_j
\]
Note that $\x^\top \F_{\D_{\mathscr{R}}(\mathfrak{a})}(\hat{\jv})\x=\mathfrak{a}^\top \mathfrak{v}(\x)$ for $\x \in \jv^\perp$. 
Since $\vf (\x)=(\lambda_{1j},\ldots, \lambda_{sj})$ for each $\x \in E_j$ with $\x^\top \x=1$,  the eigenvalue of 
$\F_{\D_{\mathscr{R}}(\mathfrak{a})}(\hat{\jv})$ on $E_j$
is expressed as $\sum_{i=1}^s  a_i  \lambda_{ij}=\mathfrak{a}^\top \mathfrak{v}(\x)$. 
This implies that signature $(p,q)$ of $\F_{\D_{\mathscr{R}}(\mathfrak{a})}(\hat{\jv})$ is the desired signature in the theorem.
\end{proof}
\begin{remark}
    Let $(V,\mathscr{R})$ be a symmetric association scheme of class $s$. 
Let $\E_i$ be the orthogonal projection matrix onto eigenspace $E_i$. The matrix $\E_i$ can be identified with the Gram matrix of a representation of  $(V,\mathscr{R})$ in $\mathbb{R}^{m_i}$. From 
Theorem~\ref{thm:dimrep},  a non-Euclidean representation of a minimal dimensionality has a Gram matrix $a\E_i- b \E_j$ for $a,b>0$, where the signature of the matrix is $(m_i,m_j)$.  
\end{remark}
\begin{remark} \label{rem:2-dis}
For $s=2$, $\Pm\A_1\Pm$ and $\Pm\A_2 \Pm$ are always commutative and $\mathfrak{V}_\mathscr{R}$ is a segment on the line $\jv^\perp\subset \R^2$. From Theorems~\ref{thm:dim_PMP} and \ref{thm:dimrep}, we can confirm Theorem~2.4 in \cite{NS12} (see also \cite{R10}) on the minimal Euclidean dimension of a graph with eigenvalues and main angles. 
\end{remark}

\section{Algorithm for the classification of largest $2$-indefinite-distance sets} 
\label{sec:4}
The algorithm for 
classifying the largest proper  $2$-indefinite-distance sets in $\mathbb{R}^{p,q}$ is based on the algorithm of Lison\v{e}k for classifying the largest $2$-distance sets in $\mathbb{R}^d$.  
To find the largest 2-indefinite-distance sets, we suppose that distances $a,b$ are not 0.
We prove several lemmas to describe the proposed algorithm. 


\begin{lemma} \label{lem:subsign}
Let $\M$ be a real symmetric matrix of signature $(r,s)$. 
Then there exists a principal $(r+s) \times (r+s)$ submatrix of signature $(r,s)$. 
\end{lemma}
\begin{proof}
Applying suitable permutations to the rows and columns of $\M$, an orthogonal matrix $\Pm$ can be generated such that $\Pm^{\top} \M \Pm$ is a diagonal matrix 
\[
\begin{pmatrix}
\D & \Om \\
\Om & \Om 
\end{pmatrix},
\]
where $\D$ is an $(r+s)\times (r+s)$ diagonal matrix with $r$ positive and $s$ negative entries. 
Then 
\begin{align*}
\M=\Pm\begin{pmatrix}
\D & \Om \\
\Om & \Om 
\end{pmatrix}\Pm^\top= 
\begin{pmatrix}
\Pm_{11} & \Pm_{12} \\
\Pm_{21} & \Pm_{22} 
\end{pmatrix}\begin{pmatrix}
\D & \Om \\
\Om & \Om 
\end{pmatrix}\begin{pmatrix}
\Pm_{11}^\top & \Pm_{21}^\top \\
\Pm_{12}^\top & \Pm_{22}^\top 
\end{pmatrix}
=\begin{pmatrix}
\Pm_{11}\D \Pm_{11}^\top &\Pm_{11}\D \Pm_{21}^\top \\
\Pm_{21}\D \Pm_{11}^\top & \Pm_{21}\D \Pm_{21}^\top 
\end{pmatrix},
\end{align*}
and $\Pm_{11}\D \Pm_{11}^\top$ is an $(r+s) \times (r+s)$ submatrix with signature $(r,s)$. 
\end{proof}

\begin{lemma} \label{lem:eigen_rep_graph}
Let $G$ be a simple graph $(V,\{R_0,R_1,R_2\})$ of order $n$ with relation matrices $\A_i$. Let $p,q$ be non-negative integers such that $n\geq p+q+3$. 
 Let $\Pm=\I-(1/n)\J$.   
Then, the following are equivalent: 
\begin{enumerate}
\item A dissimilarity matrix $\D=\A_1+b \A_2$ ({\it resp}. $-\A_1  +b \A_2$) is representable in $\mathbb{R}^{p,q}$ for some $b\in \mathbb{R}$ with $|b|<1$. 
\item There exists a unique eigenvalue $\lambda >-1/2$ of $\Pm \A_1\Pm$ on $\jv^\perp$ such that the sum of the multiplicities of the eigenvalues on $\jv^\perp$ less than $\lambda$ is at most $p$ ({\it resp}. $q$), and the sum of the multiplicities of the eigenvalues greater than $\lambda$ is at most $q$ ({\it resp}. $p$). 
\end{enumerate}
Moreover, if $(1)$ holds, then $b=\lambda/(1+\lambda)$ ({\it resp}. $b=-\lambda/(1+\lambda)$) with eigenvalue $\lambda$ obtained in $(2)$. 
\end{lemma}

\begin{proof}
We only prove this lemma for $\D=\A_1+b \A_2$. 
The proof of $-\A_1+b \A_2$ is similar. 
First, we prove $(1)\Rightarrow (2)$. 
Let $(r,s)$ be the embedding dimension of $\D$. 
Then, $(r,s)={\rm sign}(-\Pm \D \Pm)$. It follows that $p\geq r$ and $q \geq s$. 
From $n\geq p+q+3\geq r+s+3$, there exists a zero eigenvalue of $-\Pm \D \Pm $ with an eigenvector $\x$ perpendicular to $\jv$. 
Since $\Pm \A_1 \Pm$ and $\Pm \A_2 \Pm$ are commutative, 
the eigenspaces of $-\Pm \D \Pm $ and $\Pm \A_i \Pm$ are identical. From $\A_2=\J-\I-\A_1$,  
we establish the following equation: 
\[
-\Pm (\A_1+b \A_2) \Pm=(b-1) \Pm \A_1 \Pm +b \Pm. 
\]
Moreover, 
the eigenvalue $\lambda$ of $\Pm \A_1 \Pm$, derived from the eigenvector $\x$, satisfies $(b-1)\lambda + b=0$, which simplifies to $b=\lambda/(1+\lambda)$. 
It is noteworthy that $|b|<1$ if and only if $\lambda>-1/2$.  
From Theorem~\ref{thm:dim_PMP} and Remark~\ref{rem:2-dis}, the sum of the multiplicities of the eigenvalues on $\jv^\perp$ less ({\it resp.} greater) than $\lambda$ is equal to $r (\leq p)$ ({\it resp.} $s (\leq q)$). Moreover, the eigenvalue $\lambda$ is unique because we can remove only one eigenspace by choosing $\lambda$. 

Considering $b=\lambda/(1+\lambda)$ for a given $\lambda$ in $(2)$, we can clearly show $(2)\Rightarrow (1)$.  
\end{proof}

Now, we determine an algorithm to classify representable graphs as 2-indefinite-distance sets in $\R^{p,q}$. All calculations were performed using Magma \cite{BCP97}. 
The algorithm is applied to two cases: $\A_1+b \A_2$ and $-\A_1+ b \A_2$ (In the case when $p=q$, we may apply  it to $\A_1+b \A_2$); however, we will outline its implementation only for $\A_1+b \A_2$. 
For Euclidean representations, a distance matrix (or a graph) is representable in $\mathbb{R}^d$ if and only if any $(d+3)$-point principal submatrix (or subgraph) of the matrix  
is representable in $\R^d$. 
This is a key argument in the Lison\v{e}k algorithm.  
 This type of result is not known for $\mathbb{R}^{p,q}$; however, if a subgraph of order $p+q+3$ is representable in $\R^{p,q}$, then the distances are uniquely determined by Theorem~\ref{thm:dim_PMP} or Lemma \ref{lem:eigen_rep_graph}. 
Moreover, any indefinite-distance matrix of $\R^{p,q}$ has a principal submatrix of size $p+q+3$ whose signature is $(p,q)$ by Lemma~\ref{lem:subsign}. 
In our algorithm, we collect all graphs whose $(p+q+3)$-vertex subgraphs are representable in $\R^{p,q}$ in a manner similar to Lison\v{e}k, and we check whether 
their largest graph is representable in $\R^{p,q}$.  
A graph is said to be {\it quasi-representable} in $\R^{p,q}$ if any $(p+q+3)$-vertex subgraph is representable in $\R^{p,q}$

First, we fixed the dimensions $(p,q)$ and collected all graphs with $p+q+3$ vertices by using Macky's algorithm \cite{M98}. 
Magma has a database of all graphs with a maximum of 10 vertices.  
For each graph with $p+q+3$ vertices, 
we calculate the characteristic polynomial $f(x)$ of $\Pm \A \Pm$ in $\mathbb{Q}[x]$. 
We obtain the prime factorization of $f(x)=g_1(x)^{m_1} \cdots g_r(x)^{m_r}$ over $\mathbb{Q}$. 
A root of $g_i(x)$ can be expressed as a sufficiently small segment $[s,t]$ with $s,t\in \mathbb{Q}$  containing the root. 
Since the roots of the irreducible polynomial $g_i(x)$ are simple, the exponent $m_i$ is the multiplicity of the root of $f(x)$. 
We collect all graphs that have an eigenvalue $\lambda$ satisfying Lemma \ref{lem:eigen_rep_graph} (2). 
Let $L[n,\lambda]$ be the set of all graphs of order $n$ that are quasi-representable in $\mathbb{R}^{p,q}$ and have an eigenvalue $\lambda\ne 0$ satisfying the condition of Lemma~\ref{lem:eigen_rep_graph} (2). 
Let $L'[n,\lambda]$ be the subset of $L[n,\lambda]$ consisting of graphs that are quasi-representable in $\mathbb{R}^{p,q}$ but not in $\mathbb{R}^{p,q-1}$ and $\mathbb{R}^{p-1,q}$.  
We can construct $L[p+q+3,\lambda]$ and $L'[p+q+3,\lambda]$ by checking for the existence of $\lambda$ that satisfies the condition of  Lemma~\ref{lem:eigen_rep_graph} (2) for each graph of order $p+q+3$. 
Here, the corresponding distance $b$ is $\lambda/(\lambda+1)$ for $\A_1+b \A_2$ and $-\lambda/(\lambda+1)$ for $-\A_1+b \A_2$. 
If $L'[p+q+3,\lambda]$ is empty for any $\lambda$, 
then there does not exist a 2-indefinite-distance set in $\mathbb{R}^{p,q}$ of size at least $p+q+3$ by Lemma~\ref{lem:subsign}.

For each graph $G$ in $L[p+q+3, \lambda]$, we add a new vertex such that the new graph is quasi-representable in $\mathbb{R}^{p,q}$. The number of options for adding a new vertex is $2^{p+q+3}$ and 
is determined by the vectors in $\{0,1\}^{p+q+3}$, where 0 in position $i$ means that the new vertex has no edge to the current vertex $i$ in the graph and 1 means that there is such an edge. 
We added a new vector to the last row and column of $\A_1$. 
For the adjacency matrix $\A'$ of the new graph after adding a new vertex, let $\A_i'$ be the principal submatrix of order $p+q+3$ by removing the $i$th row and column. 
If $\A_i'$ is in $L[p+q+3, \lambda]$ (here, we use the Magma function to check the isomorphism of the graphs) for each $i \in \{1,\ldots, p+q+3 \}$, then $\A'$ is quasi-representable in $\mathbb{R}^{p,q}$. 
By collecting all graphs $\A'$ that are quasi-representable in $\mathbb{R}^{p,q}$ in this manner, we obtain $L[p+q+4,\lambda]$. If a graph $G\in L[p+q+3,\lambda]$ is also in $L'[p+q+3,\lambda]$, then a new quasi-representable graph after adding a vertex is also added to $L'[p+q+4,\lambda]$.  
Similarly, we can construct $L[n,\lambda]$ and $L'[n,\lambda]$ for $n\geq p+q+4$ until $L'[n+1,\lambda]$ becomes empty.  
It is sufficient to check whether $\A_i'$ is in  $L[n, \lambda]$ for each $i \in \{1,\ldots, p+q+3 \}$ to determine the representability of $\A'$ (see Lison\v{e}k's argument in \cite[p.\ 336]{L97}). When $L'[n+1,\lambda]$ is empty, 
 we collect graphs in $L'[n,\lambda]$ that are representable in $\mathbb{R}^{p,q}$, which are the classification of largest proper $2$-indefinite-distance sets in $\mathbb{R}^{p,q}$. Fortunately,  all graphs in $L'[n,\lambda]$ are representable in $\R^{p,q}$ for any $p+q+3 \leq 10$. 
This algorithm is based on the Lison\v{e}k algorithm \cite{L97}, and its computational cost is almost the same as Lison\v{e}k's.

\begin{table}
    \centering
    \begin{tabular}{c|ccccccc}
         $q$ $\backslash$ $p$&1 & 2& 3& 4&  5& 6 &7\\
         \hline 
          0 &$3_1$& $5_1$ & $6_6$ & $10_1$ & $16_1$ & $27_1$ & $29_1$ \\ 
          1 & $3_{\infty}$ &$5_8$ &$7_3$&$10_2$&$13_3$& $22_1$& \\
          2 & &$7_1$ &$8_3$ & $10_3$ &$13_1$& &\\
          3 & & & $9_{14}$ &$12_1$ & & &
    \end{tabular}
    \caption{The cardinality of a largest proper 2-indefinite-distance set in $\mathbb{R}^{p,q}$}
    \label{tab:my_label}
\end{table}

Table~\ref{tab:my_label} lists the cardinality of a largest proper 2-indefinite-distance set in $\mathbb{R}^{p,q}$. 
The subscript indicates the number of the largest proper 2-indefinite-distance sets up to isomorphism. In Appendix \ref{sec:A}, we present the largest proper 2-indefinite-distance sets and their properties.

\section{Largest spherical 2-indefinite-distance sets} \label{sec:5}
We define a sphere $S_{p,q}(r)$ of radius $r>0$ in $\mathbb{R}^{p,q}$ as follows:  
\[
S_{p,q}(r)=\{ \x \in \mathbb{R}^{p,q} \mid \langle\langle \x, \x \rangle \rangle=r \}. 
\]
The sphere $S_{p,q}(1)=S_{p,q}$ is called the {\it unit sphere}. 
For $X \subset S_{p,q}$, $X$ is called a {\it proper spherical set} if $X$ is not included in $S_{p-1,q}$ or $S_{p,q-1}$. 
If an $s$-indefinite-distance set $X$ in $S_{p,q}$ satisfies $\la \la \x,\y \ra \ra \ne 1$ for any distinct $\x,\y \in X$, then we can prove an upper bound \[
|X|\leq \binom{p+q+s-1}{s}+\binom{p+q+s-2}{s-1}
\]
in the same manner as in the proof of Theorem \ref{thm:absolute_bound} using a polynomial 
\[
F(\x, \y)=\prod_{\alpha \in B(X)} \frac{\la \la \x , \y \ra \ra - \alpha}{1-\alpha},
\]
where $B(X)=\{\la\la \x, \y \ra \ra \mid \x, \y \in X, \x, \ne \y \}$. 
Improved upper bounds are known for $q=1$ \cite{BT}. 
\begin{theorem}\label{thm:ab_bound_sphere}
Let $X$ be an $s$-indefinite-distance set in $S_{p,1}$. 
If $\langle\langle \x, \y \rangle \rangle\ne 1$ for any two distinct points $\x,\y \in X$, 
then 
\[
|X| \leq \binom{p+s}{s}. 
\]
\end{theorem}
For each $p\geq 10$, 
there exists a spherical 2-indefinite-distance set attaining the upper bound from Theorem~\ref{thm:ab_bound_sphere} \cite{L97}.  
The construction is based on Lison\v{e}k \cite{L97}, and this particular construction is a representation of a graph in the switching class of the Johnson graph. 

A dissimilarity matrix $\D$ of embedding dimension $(p,q)$ is said to be {\it Type (i)} if 
$-\D$ satisfies the condition in Theorem~\ref{thm:dim_PMP} (i). 
\begin{remark} \label{rem:type}
We can determine the type of a dissimilarity matrix $\D$ of embedding dimension ${\rm sign}(\F_{\D}(\el))=(p,q)$
using the signature of $-\D$.  Thus, we obtain the following: 
\begin{enumerate}
    \item  ${\rm sign}(-\D)=(p+1,q+1)$ if and only if $\D$ is of Type (1). 
    \item ${\rm sign}(-\D)=(p,q+1)$ if and only if $\D$ is of Type (2). 
    \item ${\rm sign}(-\D)=(p+1,q)$ if and only if $\D$ is of Type (3). 
    \item ${\rm sign}(-\D)=(p,q)$ if and only if $\D$ is of Type (4). 
\end{enumerate}
\end{remark}
We can determine whether an indefinite-distance matrix has a spherical representation in its embedding dimension as follows:   
\begin{theorem} \label{thm:spherical}
Let $\D$ be a dissimilarity matrix of embedding dimension $(p,q)$.
Then, $\D$ is the indefinite-distance matrix of some set in $S_{p,q}(r)$ if and only if $\D$ is Type (2). 
\end{theorem}
\begin{proof}
 
Since the embedding dimension of $\D$ is $(p,q)$, there is $X\subset \mathbb{R}^{p,q}$ such that $\D=(||\x -\y||)_{\x, \y \in X}$. 
From the equality $||\x-\y||=||\x||+||\y||- 2\la \la \x, \y \ra \ra $, $X$ is in $S_{p,q}(r)$ for some $r>0$
if and only if the signature of $2 (\la\la \x,\y \ra\ra)_{\x, \y \in X} =
-\D +2 r \J$ is $(p,q)$ for some $r$. Therefore we consider the signature $-\D +a \J$ with $a>0$ for each Type (i). 

An eigenvalue of $-\D$ that is not main is an eigenvalue of $-\D + a \J$. Thus we consider the change of main eigenvalues from $-\D$ to $-\D+a\J$. Let $\lambda_1<\cdots<\lambda_k$ be the main eigenvalues of $-\D$, and $\mu_1< \cdots < \mu_l$ be those of $-\D +a\J$. 
From Theorem~2.2 in \cite{NS16}, it follows that 
$k=l$ and 
\[
f(x)=\prod_{i=1}^k (\mu_i-x)=\prod_{i=1}^k(\lambda_i-x)
\left( 1+ a \sum_{j=1}^k\frac{n \beta_j^2}{\lambda_j-x}\right),
\]
where $n$ is the order of $\D$. Moreover, 
$\lambda_1<\mu_1<\cdots<\lambda_k<\mu_k$ holds. 

Suppose $\D$ is of Type (1). Then, the signature of $-\D$ is $(p+1,q+1)$ and $\lambda_i \ne 0$ for each $i$ since $E_0 \subset \jv^\perp$. 
Under the condition $\sum_{j=1}^k \beta_j^2/ \lambda_j=0$, we have $f(0)=\prod_{i=1}^k \lambda_i$, which does not depend on $a$. 
For this case, ${\rm sign}(-\D +a \J)=(p+1,q+1) \ne (p,q)$. Therefore, $\D$ does not have a spherical representation in $\R^{p,q}$.

Suppose $\D$ is of Type (2). Then, the signature of $-\D$ is $(p,q+1)$ and $\lambda_i \ne 0$ for each $i$ since $E_0 \subset \jv^\perp$. 
Thus, ${\rm sign}(-\D+a \J)=(p,q)$ if and only if $f(0)=0$, namely $a= -1/(\sum_{j=1}^kn \beta_j^2/\lambda_j)$. From the definition of Type (2),  $\sum_{j=1}^kn \beta_j^2/\lambda_j<0$ and we can take $a>0$. Therefore $\D$ has a spherical embedding in $\R^{p,q}$. 
Moreover, we have ${\rm sign}(-\D+a \J)=(p+1,q)$ for $a> -1/(\sum_{j=1}^kn \beta_j^2/\lambda_j)$, and ${\rm sign}(-\D+a \J)=(p,q+1)$ for $a< -1/(\sum_{j=1}^kn \beta_j^2/\lambda_j)$.

Suppose $\D$ is of Type (3). Then, the signature of $-\D$ is $(p+1,q)$, and  $\lambda_i \ne 0$ for each $i$ since $E_0 \subset \jv^\perp$. From the definition of Type (3),  $\sum_{j=1}^kn \beta_j^2/\lambda_j > 0$ and we cannot take $a>0$ such that $f(0)=0$. In this case, ${\rm sign}(-\D +aJ)=(p+1,q) \ne (p,q)$. Therefore, $\D$ does not have a spherical representation in $\R^{p,q}$.

Suppose $\D$ is of Type (4). Then the signature of $-\D$ is $(p,q)$, and  $\lambda_i = 0$ for some $i$ since $E_0  \not \subset \jv^\perp$. 
In this case, ${\rm sign}(-\D +a \J) = (p+1,q)$. Therefore, $\D$ does not have a spherical representation in $\R^{p,q}$.
\end{proof}
\begin{remark}
The dissimilarity matrix $\D$ does not have a zero entry except for diagonals if and only if the corresponding Gram matrix $-\D+a \J$ has entries $a$ only at diagonals. 
This is consistent with the conditions $0\not\in A(X)$ and $1 \not\in B(X)$ of the two upper bounds for $\R^{p,q}$ and $S_{p,q}$.
\end{remark}
\begin{remark}
The diagonal entries of an indefinite-distance matrix $\D$ are 0. This implies $-\D$ has a negative eigenvalue. Thus, a Euclidean distance matrix is of Type (1) or Type (2). 
Nozaki and Shinohara \cite{NS12} determined when a Euclidean 2-distance set is spherical based on the properties of its distance matrix (see Theorems~2.4 and 3.7 in \cite{NS12}). 
Theorem~\ref{thm:spherical} generalizes this result. 
\end{remark}

Any dissimilarity matrix can have a spherical representation with higher dimensionality.  
\begin{theorem} \label{thm:spherical2}
Let $\D$ be a dissimilarity matrix of embedding dimension $(p,q)$. Let $\lambda_i$ be the main eigenvalues of $-\D$. Then, the following hold. 
\begin{enumerate}
    \item If $\D$ is Type (1), then $-\D+ a\J$ is the Gram matrix of some set in $S_{p+1,q+1}$ for any $a>0$. 
    \item If $\D$ is Type (2), then $-\D+ a\J$ is the Gram matrix of some set in $S_{p,q+1}$ for any $0<a<1/(\sum_{j=1}^kn \beta_j^2/\lambda_j)$, in $S_{p,q}$ for $a=1/(\sum_{j=1}^kn \beta_j^2/\lambda_j)$, and in $S_{p+1,q}$ for $a>1/(\sum_{j=1}^kn \beta_j^2/\lambda_j)$. 
    \item If $\D$ is Type (3), then $-\D+ a\J$ is the Gram matrix of some set in $S_{p+1,q}$ for any $a>0$. 
    \item If $\D$ is Type (4), then $-\D+ a\J$ is the Gram matrix of some set in $S_{p+1,q}$ for any $a>0$.
\end{enumerate}
\end{theorem}
\begin{proof}
This theorem is proved in the proof of Theorem~\ref{thm:spherical}. 
\end{proof}


To determine the largest proper $2$-indefinite-distance sets in $S_{p,q}$, we collect largest $2$-indefinite-distance sets in $\R^{p,q}$ of Type (2), $\R^{p-1,q}$ of Types (3) and (4), and $\R^{p-1,q-1}$ of Type (1). 
For 2-indefinite-distance sets,  we verify the signature of $\D=\A_1+(\lambda/(1+\lambda))\A_2$ (or $-\A_1-(\lambda/(1+\lambda))\A_2$) to determine whether the representation in the embedding dimension is spherical.  
The eigenvalue $\lambda$ is an algebraic number that can be expressed by the irreducible polynomial $f_{\lambda}(x) \in \mathbb{Q}[x]$ with $f_\lambda(\lambda)=0$ and an interval $[a,b]$ with $a,b\in \mathbb{Q}$ containing $\lambda$. The characteristic polynomial $g(x)$ of $-\D$ is a polynomial over 
$\mathbb{Q}[\lambda]$. 
According to Descartes’ rule of signs, if the coefficients $g_i(\lambda)$ of $x^i $ in $g(x)$ are determined, then we can know the signature of $-\D$. 
Indeed, if we can take a rational interval $[a,b]$ containing $\lambda$ such that $[a,b]$ is disjoint from an interval containing each root of $g_i(x)$, then the sign of $g_i(a)$ is the same as $g_i(\lambda)$. 
From Remark \ref{rem:type} and Theorem \ref{thm:spherical2}, we can obtain the smallest spherical dimension where the graph can be representable. 
Using the collections of representable graphs calculated in Section \ref{sec:4}, we can determine the largest proper spherical 2-indefinite-distance sets.   
Note that for $p=q$, if $\D$ is of Type (3), then $-\D$ is of Type (2) from Remark \ref{rem:type}. 

Using this algorithm, we can determine whether larger proper 
2-indefinite-distance sets obtained in the previous section are spherical
in the embedding dimension. See the appendix section \ref{sec:B}. 
Table~\ref{tab:2} lists the size and number of the largest spherical proper 2-indefinite-distance sets in $\R^{p,q}$. 

\begin{table}
    \centering
    \begin{tabular}{c|ccccccc}
         $q$ $\backslash$ $p$&1 & 2& 3& 4&  5& 6 &7\\
         \hline 
          0 &$3_1$& $5_1$ & $6_6$ & $10_1$ & $16_1$ & $27_1$ & $28_1$ \\ 
          1 &$3_{\infty}$ &$4_{\infty}$ &$7_3$&$10_1$&$13_3$& $22_1$& \\
          2 & & $7_1$ &$8_3$ & $10_3$ &$13_1$& &\\
          3 & & & $9_{14}$ &$11_3$ & & &
    \end{tabular}
    \caption{The cardinality of a largest proper spherical 2-indefinite-distance set in $\mathbb{R}^{p,q}$}
    \label{tab:2}
\end{table}

\section{Concluding remarks}
In this section, we discuss several factors related to $s$-indefinite-distance sets in $\R^{p,q}$. 
\begin{remark}
From the observation of examples obtained by PC calculation, we construct an infinite series of 2-indefinite-distance sets in $\R^{p,1}$ that have large cardinalities. 
We consider graph $G=(V,E)$, where $V=V_1\cup V_2$, $E=E_1\cup E_2 \cup E_3$, and
\begin{align*}
V_1&=\{v_{i,1}:1\le i\le n\} \cup \{v_{i,2}: 1\le i\le n\},\\
V_2&=\{w_{i,j}:1\le i<j\le n\},\\
E_1&=\{\{v_{i,1},v_{j,2}\}:i\ne j\},\\
E_2&=\{\{w_{i,j},w_{k,l}\}: |\{i,j\} \cap \{k,l\}|=\emptyset\},\\
E_3&=\{\{v_{i,j},w_{k,l}\}:i \in \{k,l\}\}.
\end{align*}
The induced subgraph of $G$ on $V_1$ is obtained from the complete bipartite graph by removing a perfect matching, and the induced subgraph of $G$ on $V_2$ is the complement of the Johnson graph $J(n,2)$. 
The representation of $G$ with distances $a=4,b=2$ is in $\R^{n,1}$ for $n\geq 7$, in $\R^{n}$ for $n=6$, and $\R^{n+1}$ for $n=3,4,5$. For $n\geq 7$, 
we provide the coordinates of a representation $X$ of $G$  
as follows.  
Let 
\begin{align*}
x_{i,1}&=\frac{3}{n}\sum_{j=1}^n{\ev}_j-{\ev}_i+c_1 {\ev}_{n+1}+c_2 {\ev}_{n+2},\\
x_{i,2}&=\frac{3}{n}\sum_{j=1}^n{\ev}_j-{\ev}_i+c_3 {\ev}_{n+1}+c_4 {\ev}_{n+2},\\
y_{i,j}&={\ev}_i+{\ev}_j, 
\end{align*}
where ${\ev}_i$ is the vector with 1 in $i$-th component and 0 in all others.  
\[c_1=\frac{3}{4}, c_2=\frac{\sqrt{n(25n-144)}}{4n}, 
c_3=\frac{3(2n-9)+4\sqrt{3n(n-6)}c_2}{4(n-9)}, c_4=\frac{(c_1c_3+2)n-9}{nc_2}\]
if $n\ne 9$, and 
\[c_1=c_2=\frac{3}{4}, c_3=-\frac{2}{3}, c_4=\frac{2}{3}\]
if $n=9$. 
Then, the representation of $G$ is
\[X=\{x_{i,1}\mid 1\le i\le n\} \cup \{x_{i,2}\mid 1\le i\le n\} \cup 
\{y_{i,j}\mid 1\le i< j\le n\}\subset \R^{n+1,1} \]
with correspondences $v_{i,j} \mapsto x_{i,j}$ and $w_{i,j}\mapsto y_{i,j}$. 
Since the sum of the first $n$ coordinates of each vector in $X$ is equal to 2, 
$X$ can be interpreted as a subset of $\R^{n,1}$. 
Their cardinalities are $|V|=2n+n(n-1)/2=n(n+3)/2$, and the value is 1 fewer than the upper bound of Theorem \ref{thm:ab_bound_sphere}. For $p\geq 10$, there exist tight sets \cite{L97}, and for $p=7,8,9$, the examples constructed here are the candidates for the largest $2$-indefinite-distance sets in $S_{p,1}$. The set for $n=6$ is the largest $2$-distance set in $\R^6$ \cite{L97}. 
\end{remark}

\begin{remark}
The largest $2$-indefinite-distance set in $\R^{6,1}$ (see Subsection \ref{sec:A12}) has a Johnson graph structure and the complete graph. The graph structure can be naturally generalized. Apparently, we can obtain 2-indefinite-distance sets in $\R^{p,1}$ for $p\geq 7$ with $1+p+p(p-1)/2$ points. We can rediscover the known largest 2-distance set in $\R^5$ with 16 points in this series. 
\end{remark}

\begin{remark}
We could not determine the largest proper 2-indefinite-distance sets in $\R^{1,1}$ with $0 \in A(X)$. 
For the case $0 \in A(X)$, it is difficult to consider representations of a dissimilarity matrix because we cannot determine whether the map of a representation is bijective.  
\end{remark}

\begin{remark}
If a dissimilarity matrix $a \A_1+b \A_2$ with $a<0$ and $b<0$ is representable in $\R^{p,1}$, then the cardinality is small.  
 If the order of the corresponding graph is at least the Ramsey number $R(3,3)=6$, then there are cliques or co-cliques of size 3 in the graph and $q\geq 2$. This implies that the order of the graph is smaller than $R(3,3)=6$. 
\end{remark}

\begin{remark}
Let $X$ be an $s$-indefinite-distance set in $\R^{p,q}$ with $0 \not\in A(X)=\{\alpha_1,\ldots , \alpha_s \}$. 
If  the cardinality of $X$ is at least  $2\binom{p+q+s-1}{s-1}+2 \binom{p+q+s-2}{s-2}$, then 
\[
K_i=\prod_{j=1,\ldots, s: j\ne i}
\frac{\alpha_j}{\alpha_j-\alpha_i}
\]
is an integer for any $i \in \{1,\ldots, s\}$, and $|K_i|$ is bounded above by a value that depends only on $p,q,s$.  This is proven in the same manner as in \cite{N11}. For $s=2$ and $p+q\leq 7$, the upper bound on $|K_i|$ is $|K_i| \leq 2$; that is,  we may normalize $(a,b)=\pm (1,1/2)$ ($|a|>|b|$). 
If the size of a largest proper 2-indefinite-distance set in $\mathbb{R}^{p,q}$ is at least $2(p+q)+4$, then we do not have to deal with other distances $b$. However, the size of a largest proper 2-indefinite-distance is small except for $(p,q)=(6,1)$, and we must check all the possibilities of distance $b$. 
\end{remark}

\bigskip

\noindent
\textbf{Acknowledgments.} 
The authors would like to thank the anonymous referees who provided valuable suggestions for improving the initial version of our manuscript.
 Nozaki is supported by JSPS KAKENHI Grant Numbers  19K03445 and 20K03527. 
Shinohara is supported by JSPS KAKENHI Grant Numbers 18K03396 and 22K03402. 
Suda is supported by JSPS KAKENHI Grant Numbers 18K03395 and 22K03410.

\quad 

\noindent
\textbf{Declarations}

\quad 

\noindent 
\textbf{Conflict of Interest}
The authors state
that there is no conflict of interest.

\appendix

\section{Figures or matrices of largest proper 2-indefinite-distance sets}
\label{sec:A}
\subsection{$(p,q)=(1,1)$}
From the algorithm written in Section \ref{sec:4}, there exists no proper 2-indefinite-distance set of 5 points. By applying Theorem~\ref{thm:dimrep} to all graphs of order 4, we can prove there exists no proper 2-indefinite-distance set of 4 points. For the dissimilarity matrix 
\begin{equation} \label{eq:(1,1)}
\M_1=\begin{pmatrix}
0&1&0\\
1&0&0\\
0&0&0
\end{pmatrix}+
b\begin{pmatrix}
0&0&1\\
0&0&1\\
1&1&0
\end{pmatrix}, (\text{{\it resp. }} \M_2=-\begin{pmatrix}
0&1&0\\
1&0&0\\
0&0&0
\end{pmatrix}+
b\begin{pmatrix}
0&0&1\\
0&0&1\\
1&1&0
\end{pmatrix} )
\end{equation}
if $b<1/4$ ({\it resp.} $b>-1/4$) holds, then the embedding dimension of $\M_1$ ({\it reps.} $\M_2$) is $(1,1)$. There are infinitely many proper 2-indefinite-distance sets of 3 points in $\R^{1,1}$. 
\subsection{$(p,q)=(2,1)$, 8 graphs on 5 vertices}

\includegraphics[width=3cm]{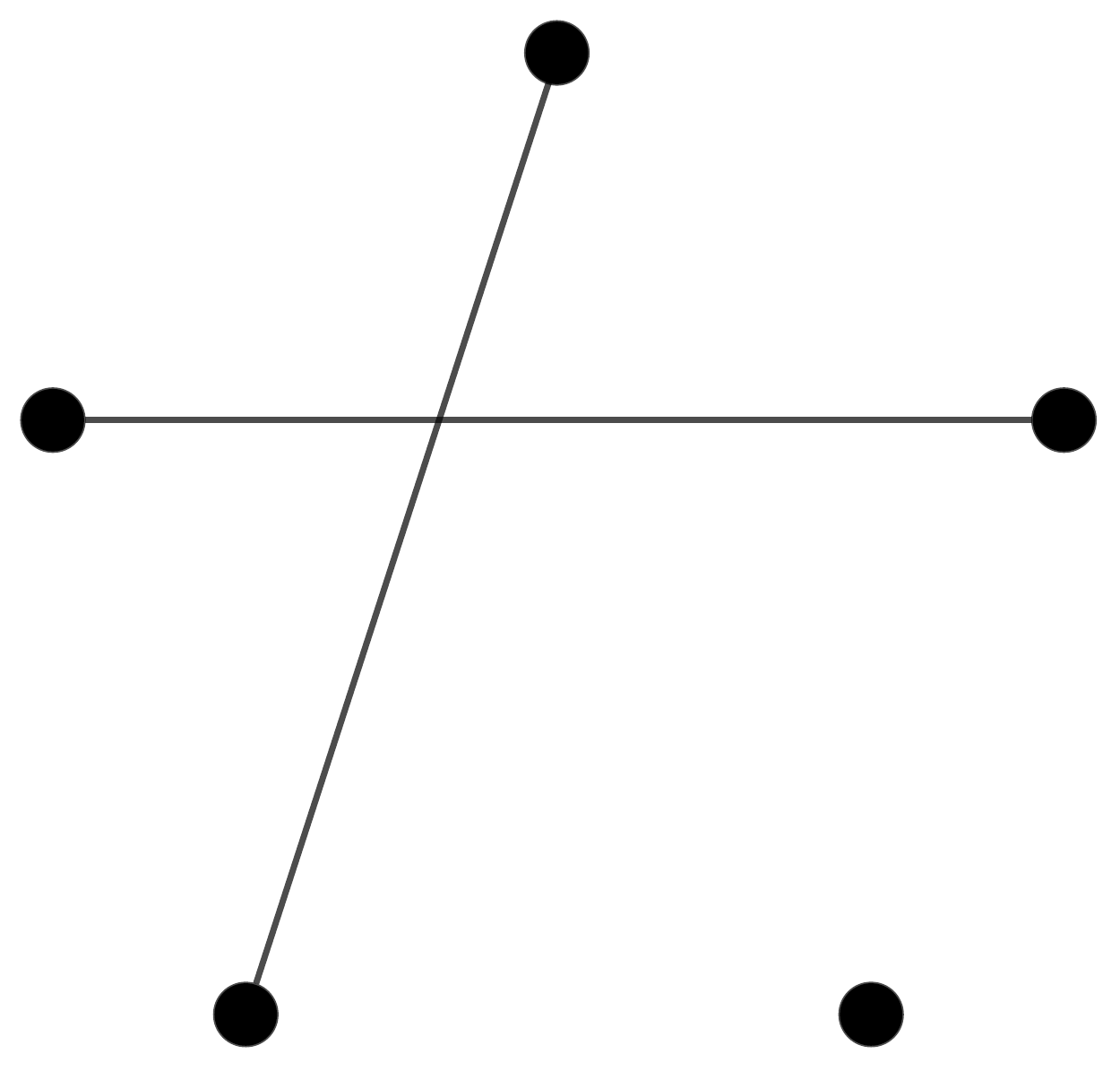}\hspace{1cm}
\includegraphics[width=3cm]{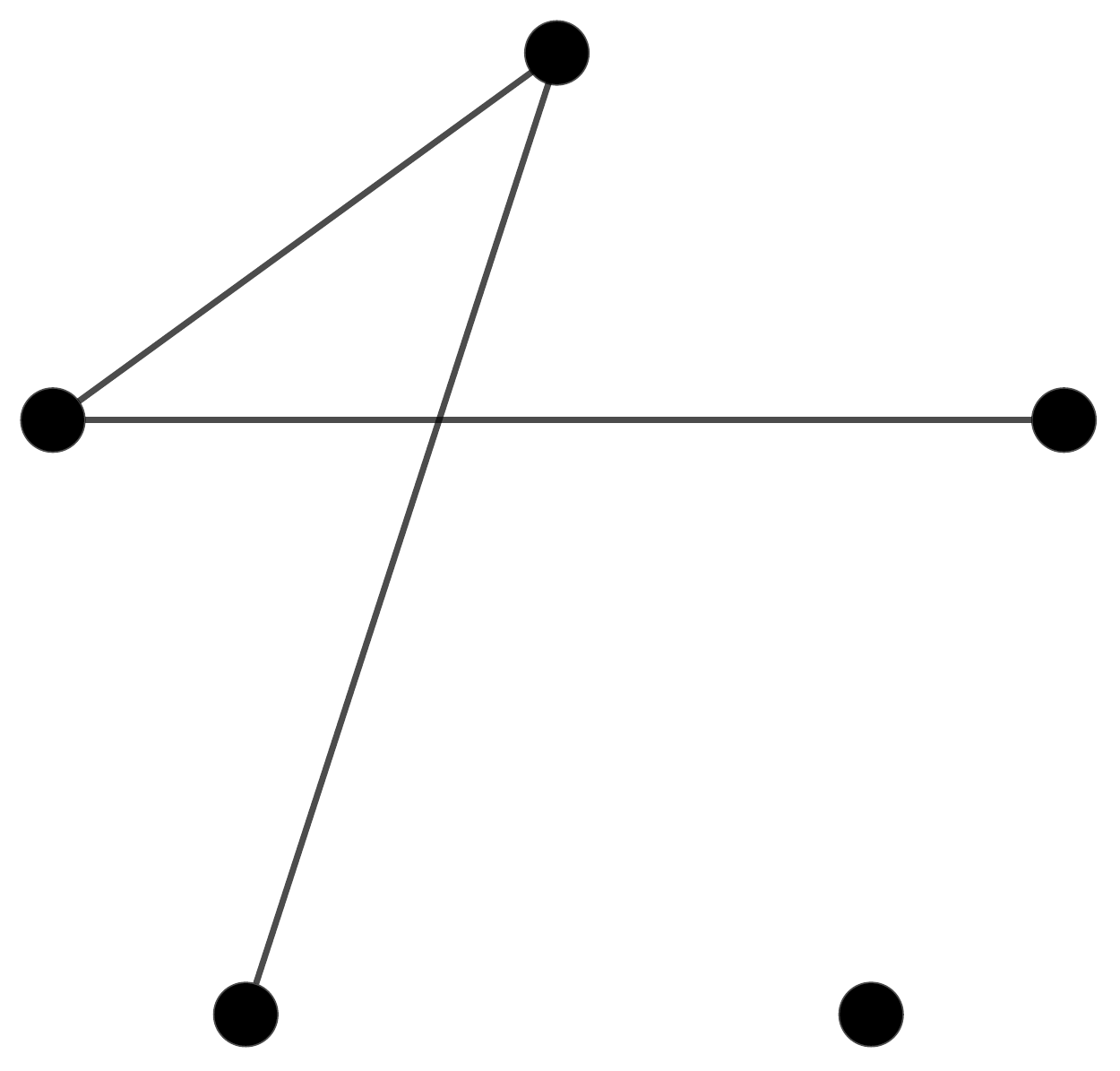}\hspace{1cm}
\includegraphics[width=3cm]{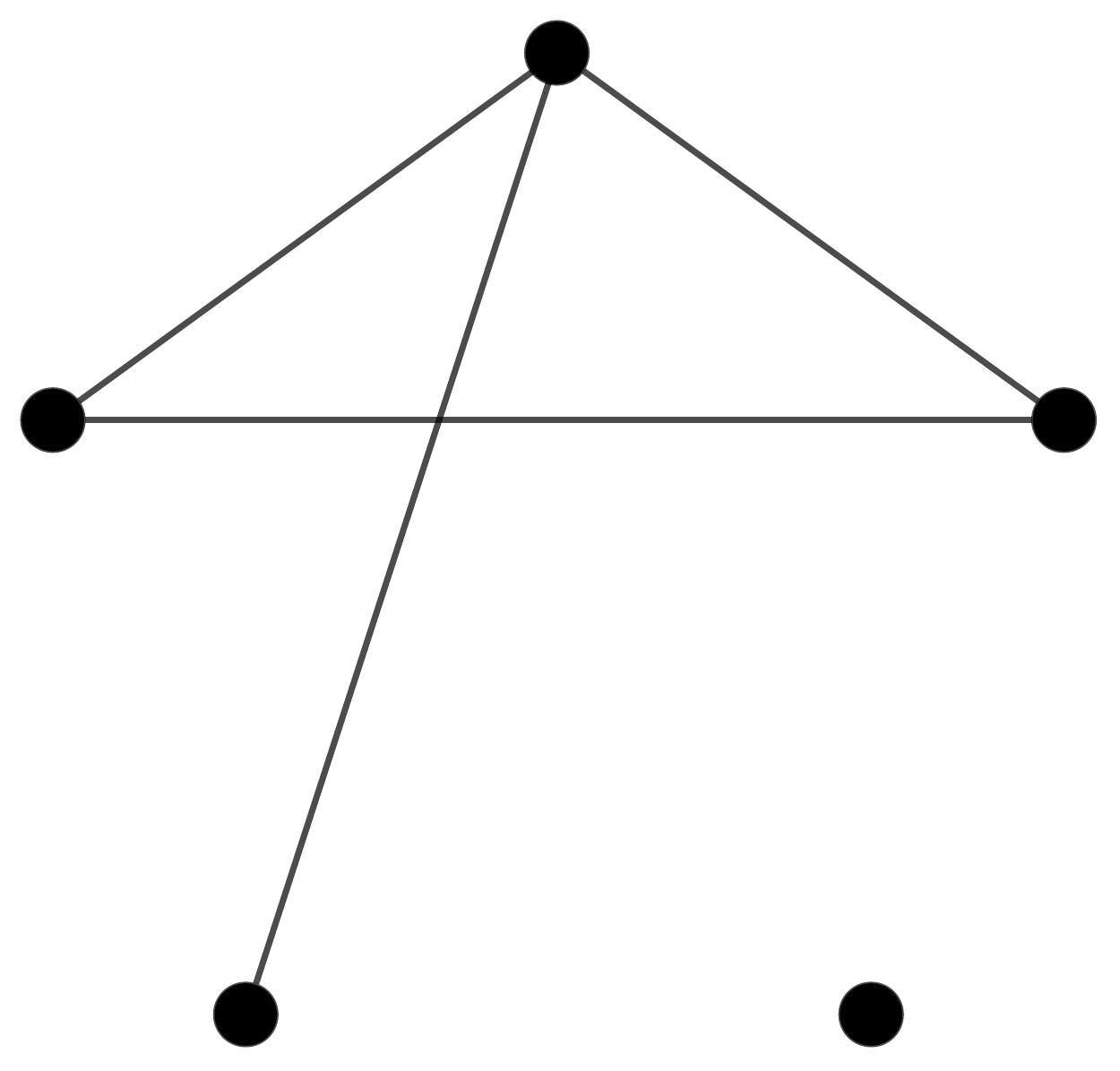}\hspace{1cm}
\includegraphics[width=3cm]{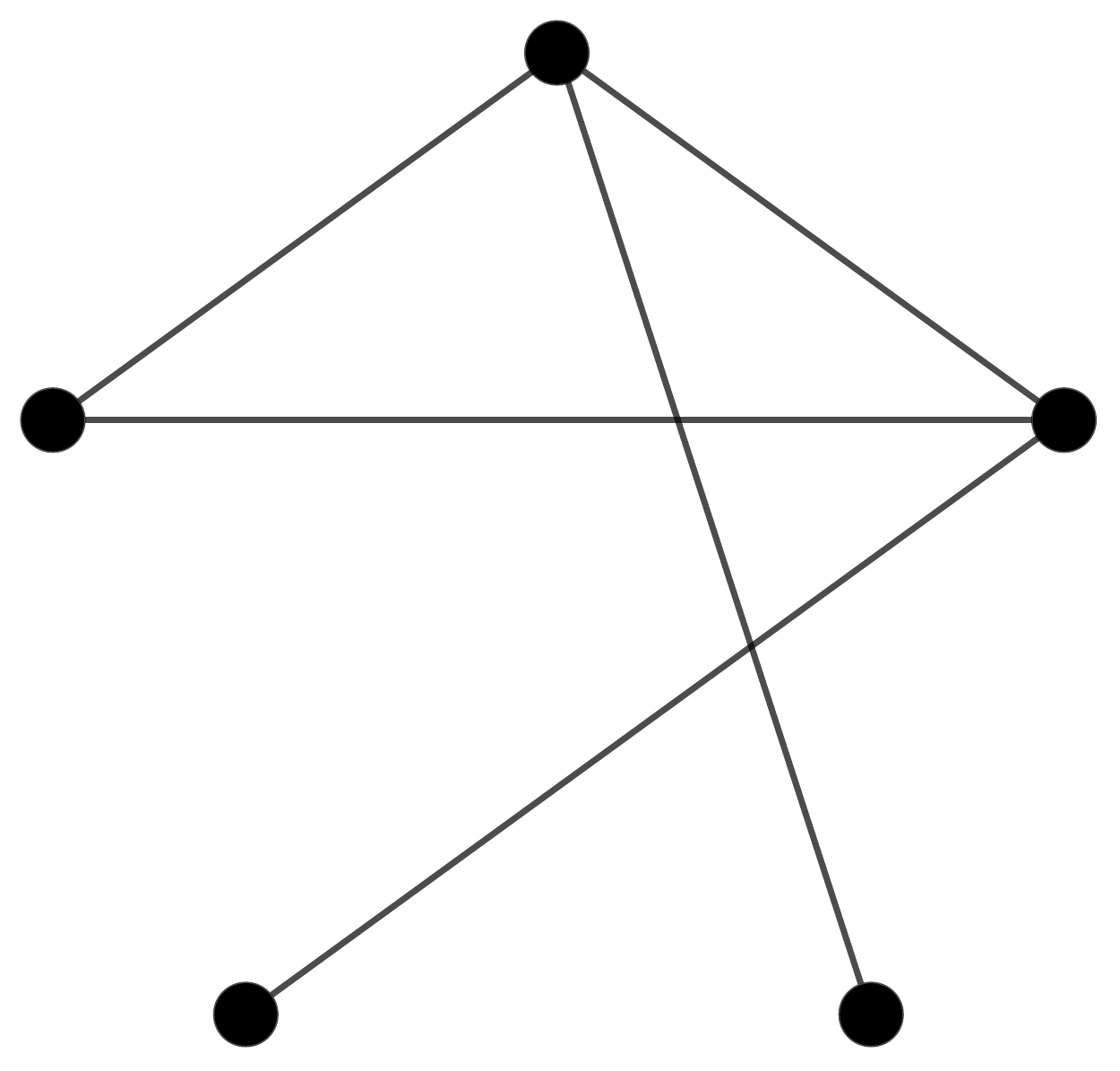}\hspace{1cm}
\\ \quad \\
\includegraphics[width=3cm]{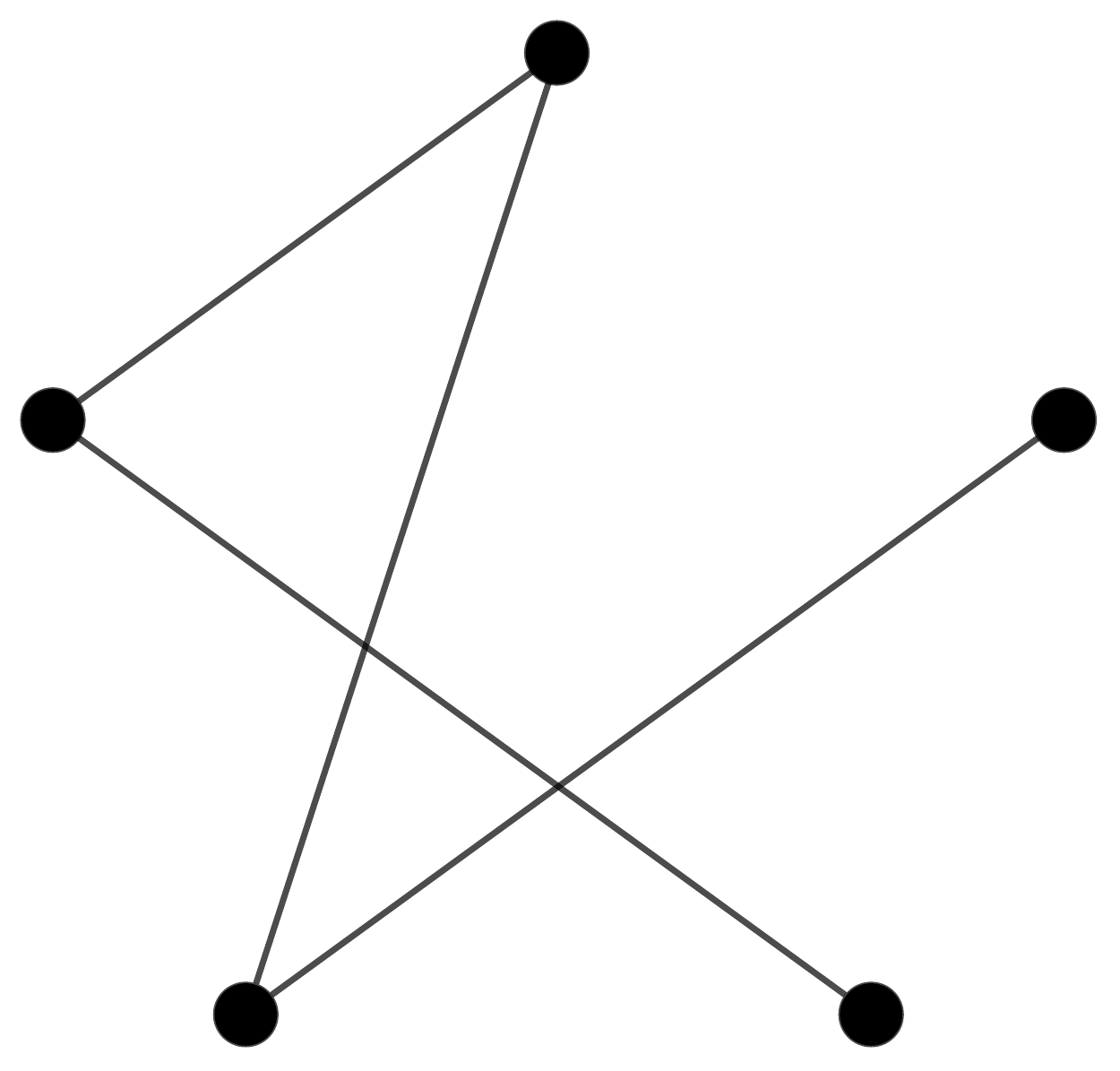}\hspace{1cm}
\includegraphics[width=3cm]{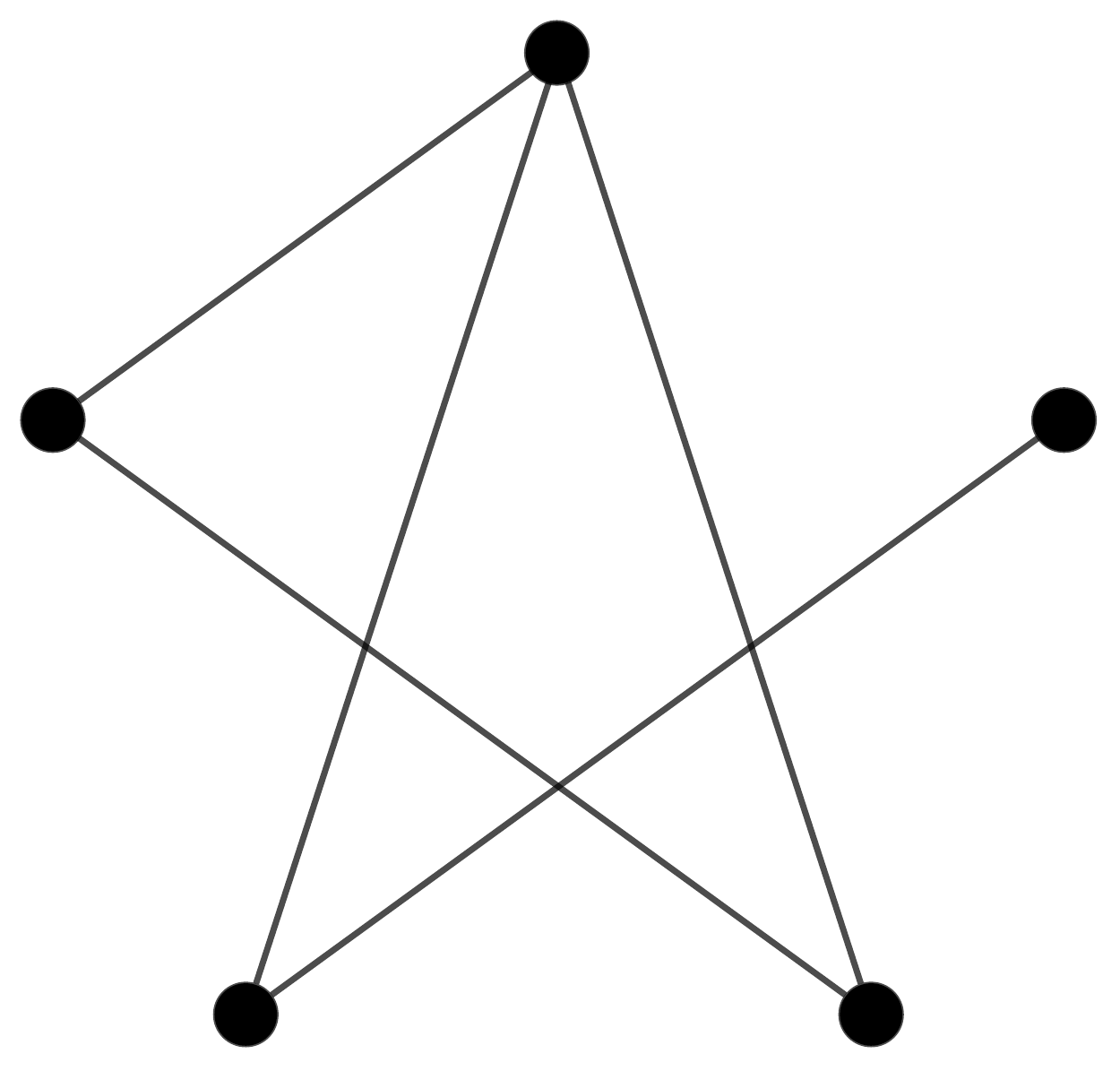}\hspace{1cm}
\includegraphics[width=3cm]{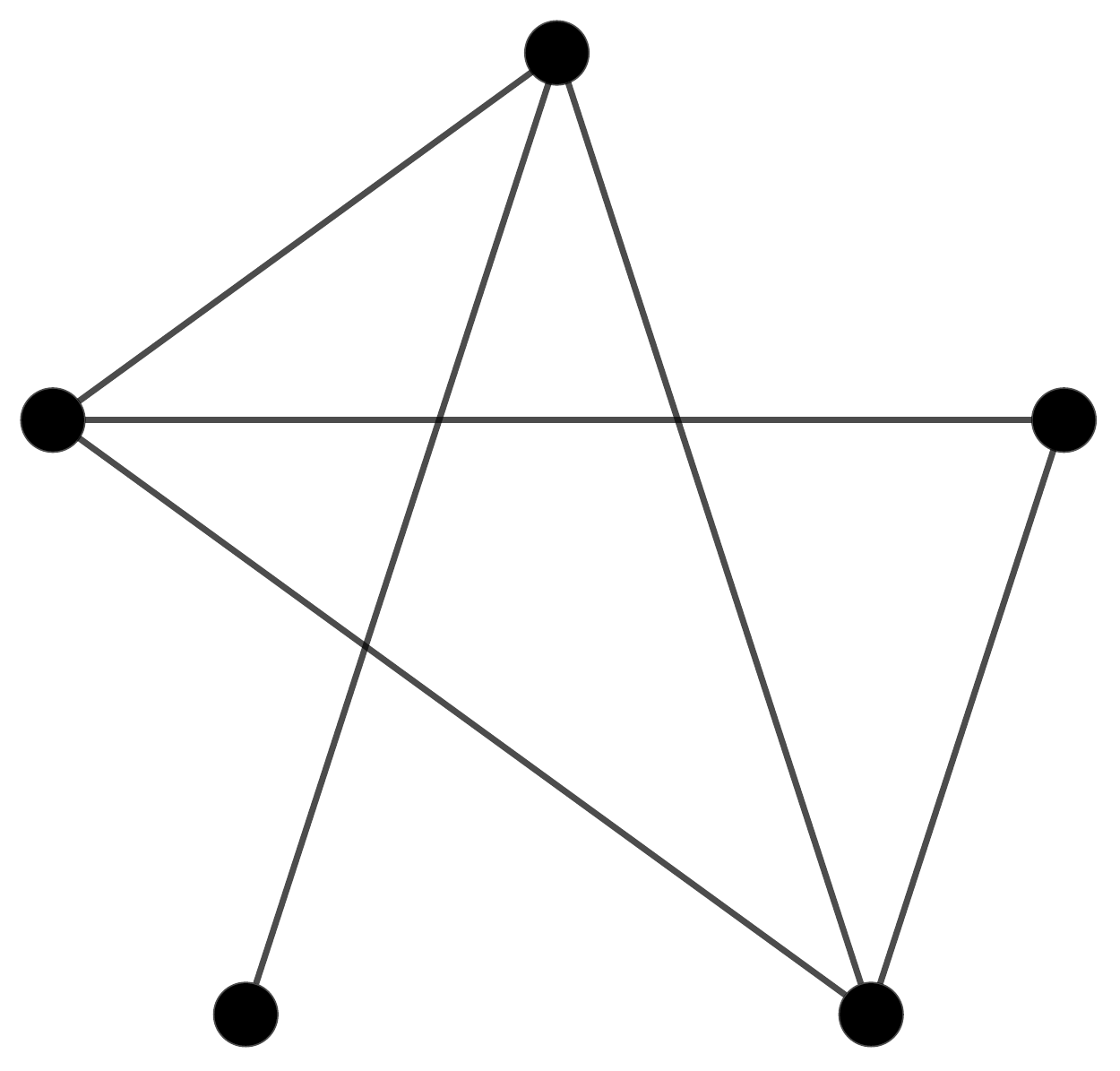}\hspace{1cm}
\includegraphics[width=3cm]{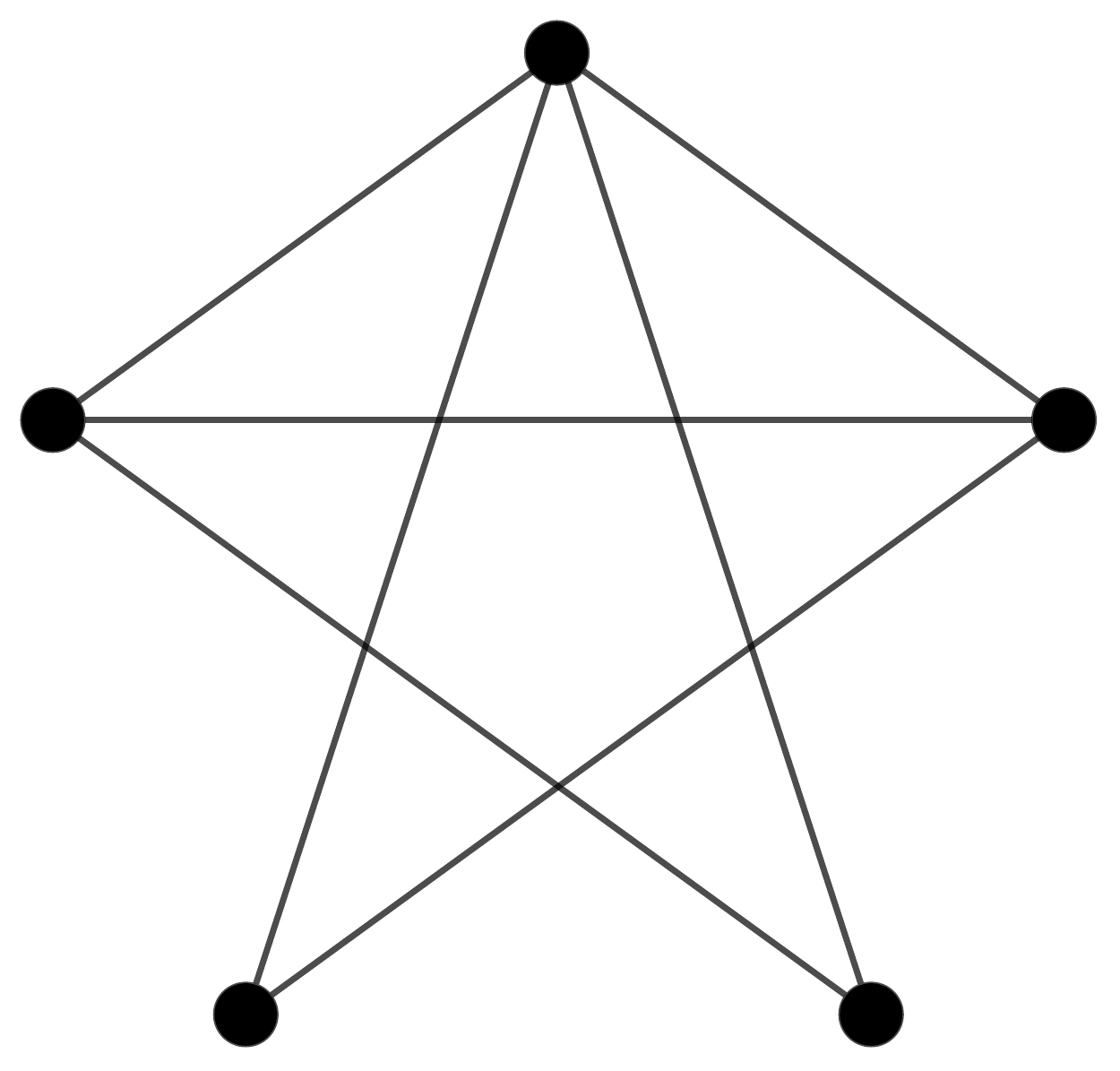}\hspace{1cm}
\\ \quad \\

For these 8 graphs, the distances are $a=1$ and $b=\lambda/(1+\lambda)$, where 
$\lambda=1/5$, the second-smallest roots of 
$x^2 + (1/5) x - 1/5$, $x^3 + (3/5)x^2 - x + 1/5$, $x^2 + x - 1/5$, $x^2 + (8/5)x - 1/5$, $x^3 + x^2 - (8/5)x - 4/5$, $x^3 + (7/5)x^2 - x - 3/5$, $x^2 + (9/5)x + 3/5$, respectively. 

\subsection{$(p,q)=(2,2)$, 1 graph on 7 vertices} 
The graph is the 7-gon and the distances are $a=1$ and $b=\lambda/(1+\lambda)$,  where $\lambda$ is the second-smallest root of $x^3+x^2-2x-1$. 

\subsection{$(p,q)=(3,1)$, 3 graphs on 7 vertices}

\includegraphics[width=3cm]{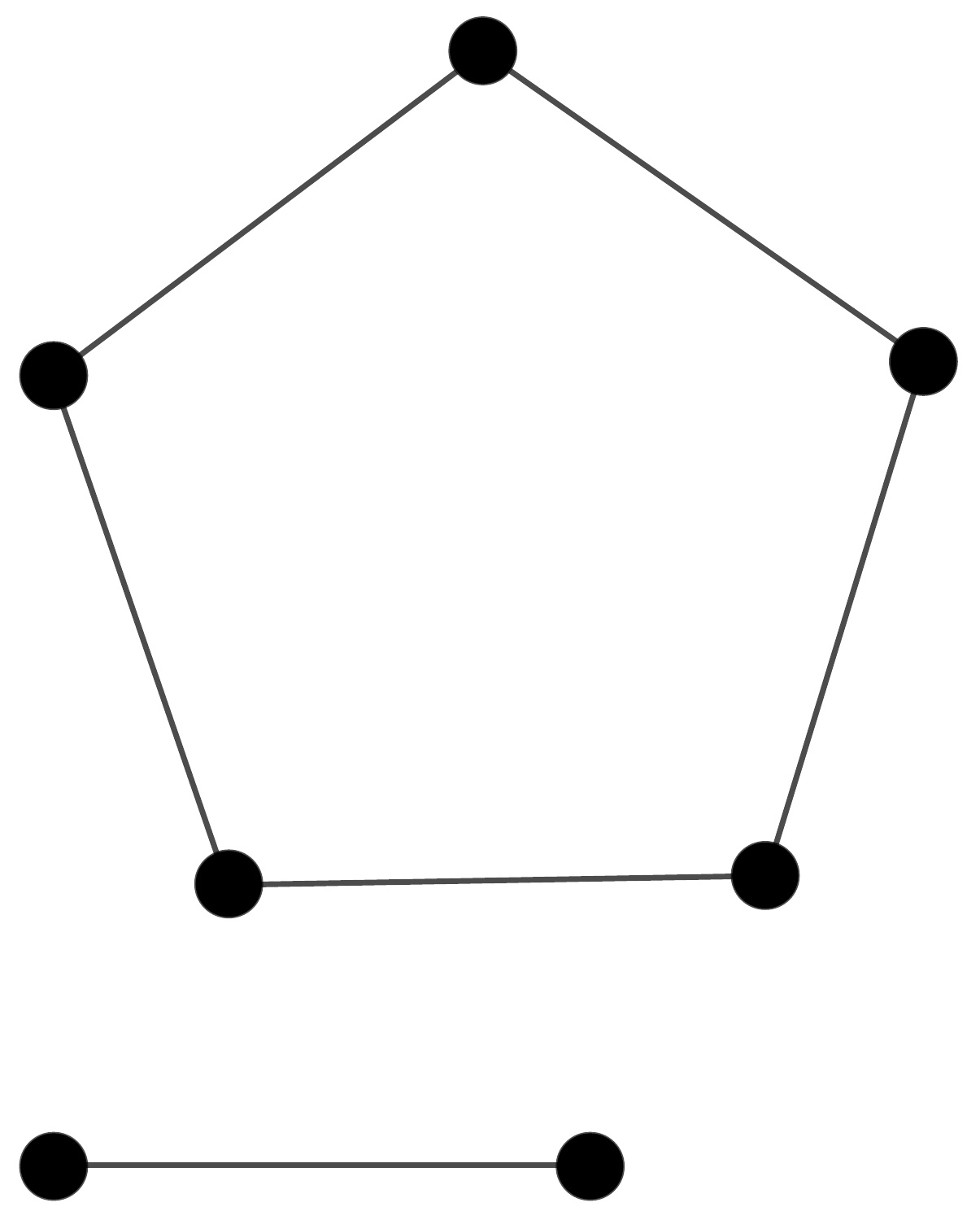}\hspace{1cm}
\includegraphics[width=3cm]{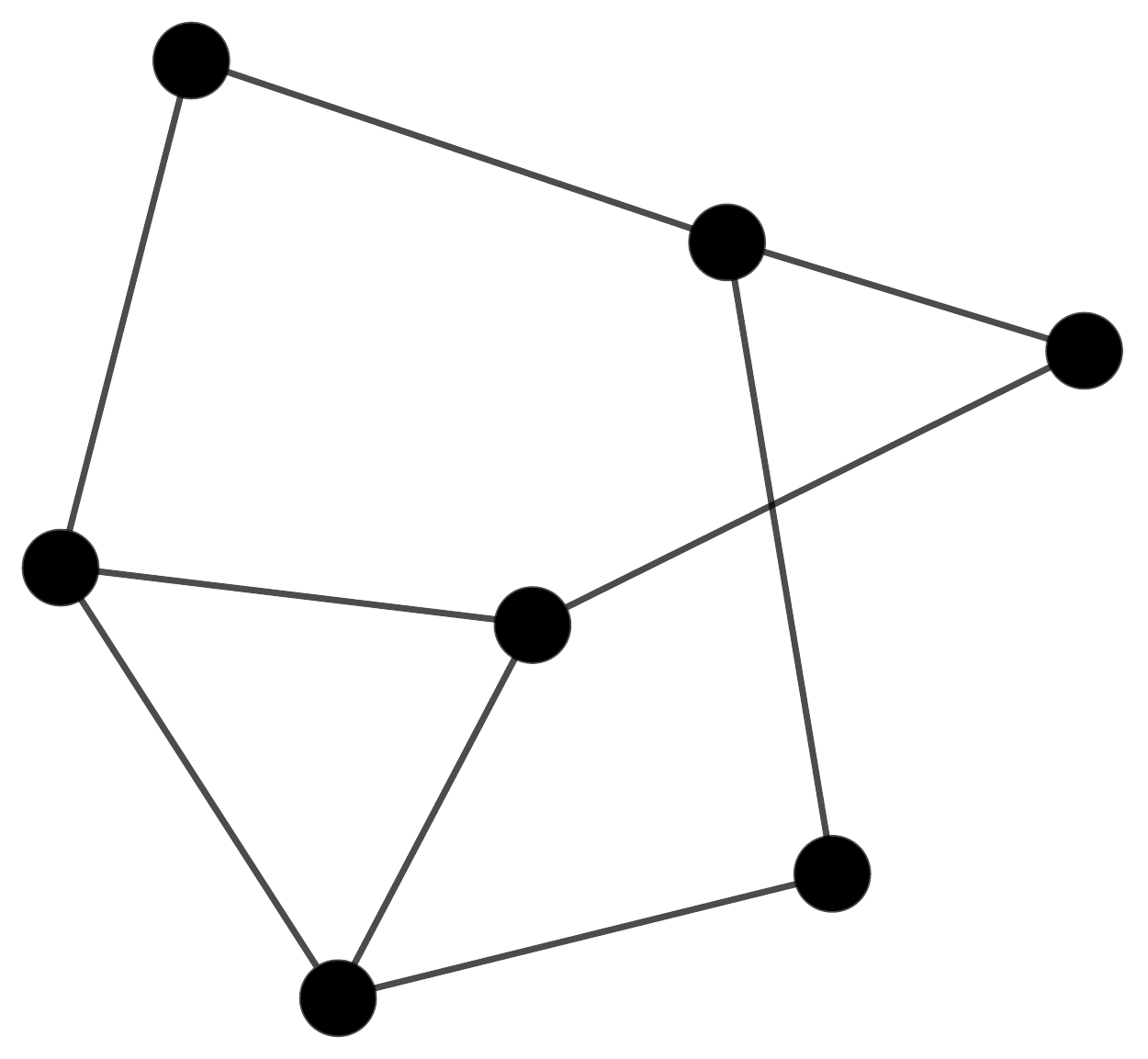}\hspace{1cm}
\includegraphics[width=3cm]{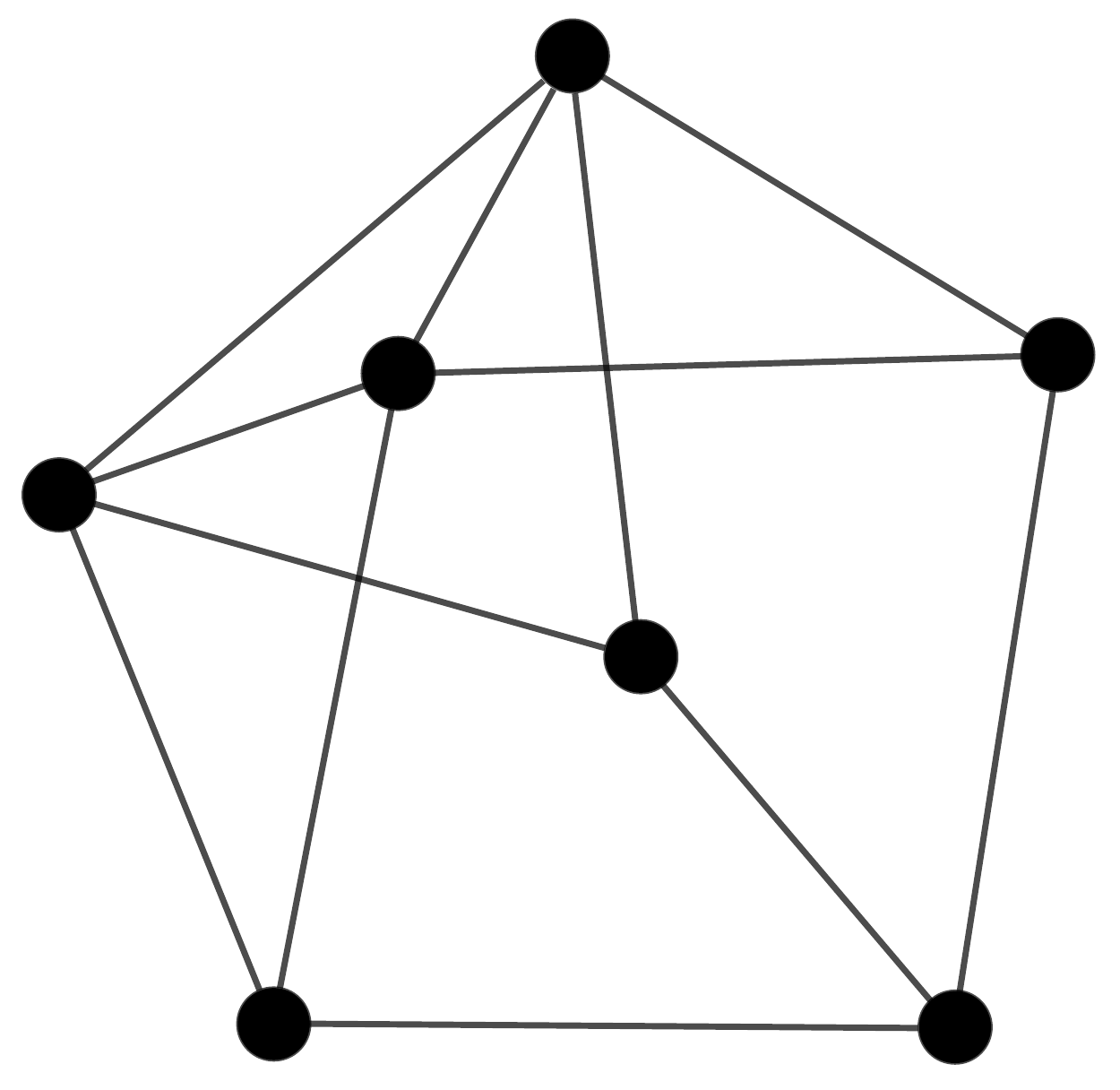}
\\ \quad \\
For the three graphs, the distances are $a=1$ and $b=\lambda/(1+\lambda)$, where $\lambda=(-1+\sqrt{5})/2$. 

\subsection{$(p,q)=(3,2)$, 3 graphs on 8 vertices}

\includegraphics[width=3cm]{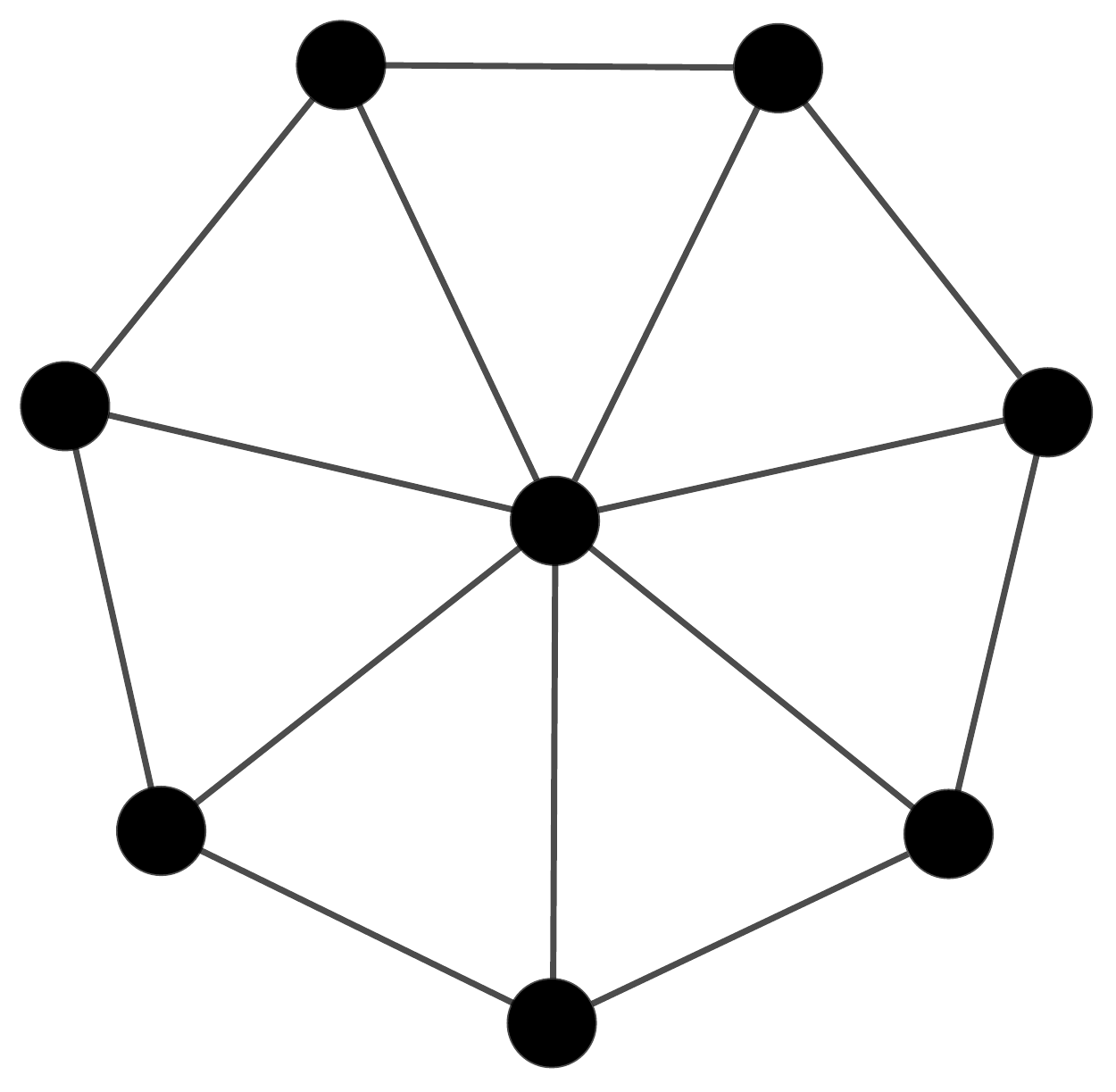}\hspace{1cm}
\includegraphics[width=3cm]{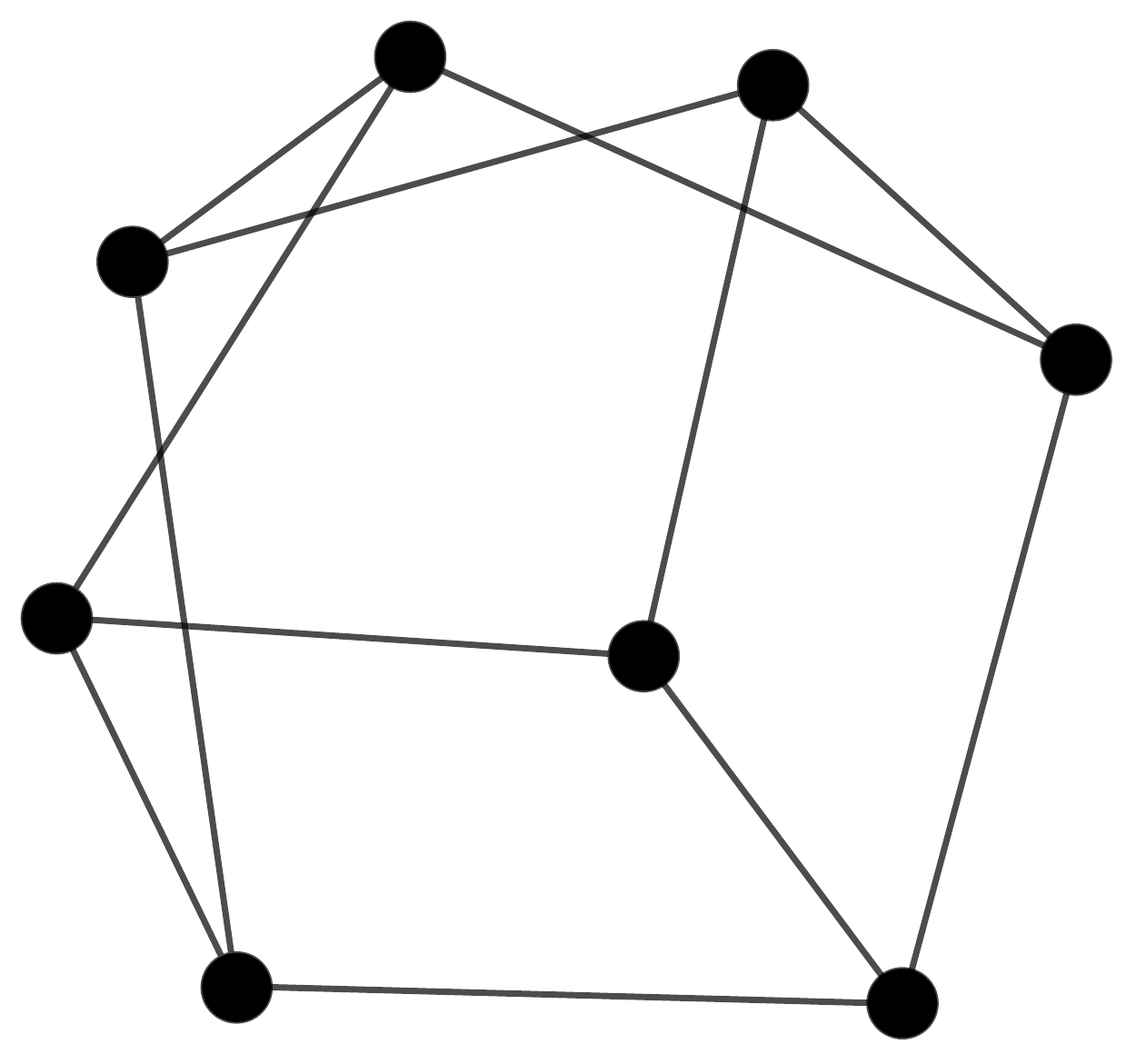}\hspace{1cm}
\includegraphics[width=3cm]{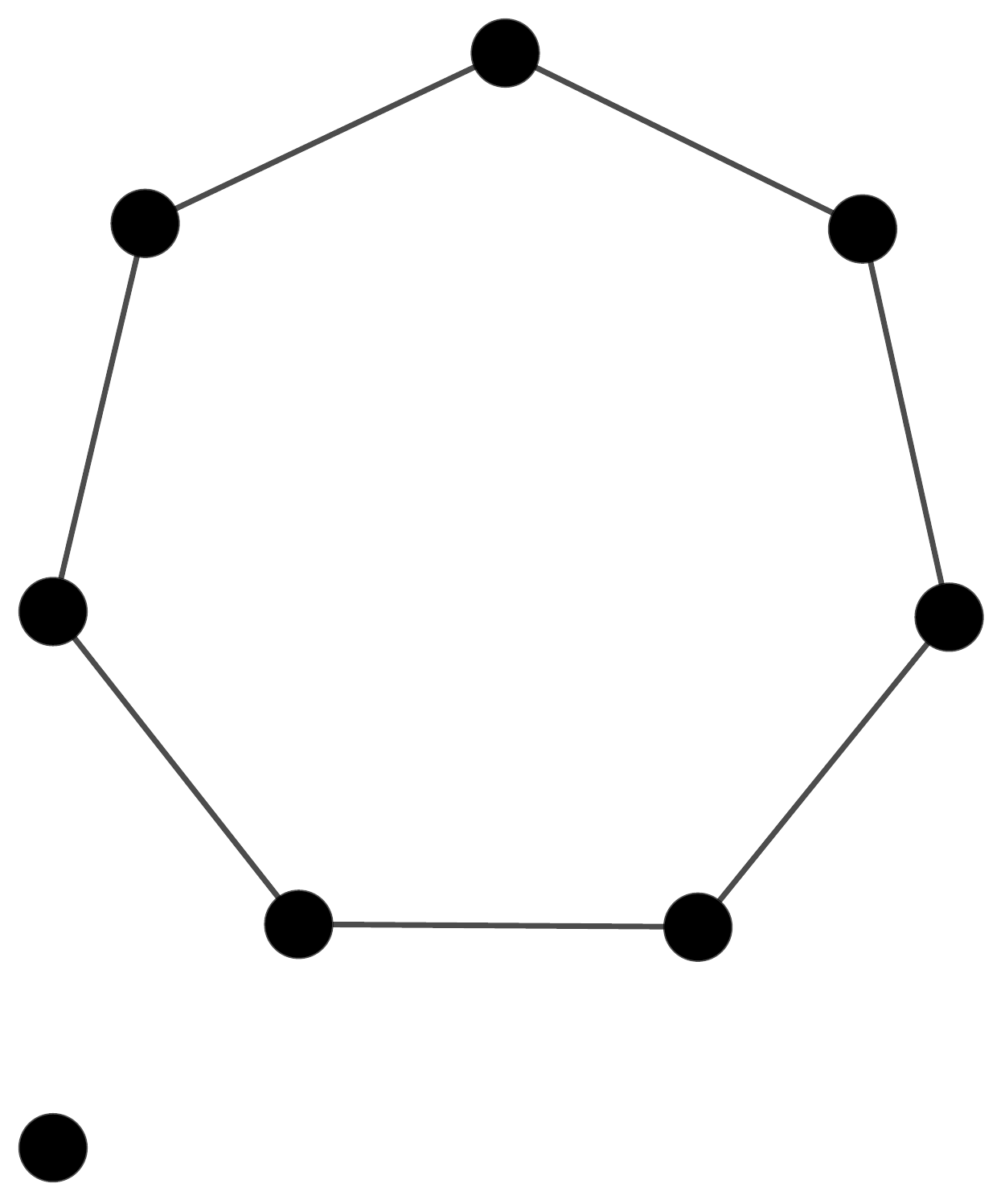}
\\ \quad \\
For the left graph, 
the distances are $a=1$ and $b=\lambda/(1+\lambda)$ where $\lambda$ is the second-smallest root of $x^3+x^2-2x-1$. 
For the center graph, 
the distances are  $a=1$ and $b=\lambda/(1+\lambda)$, where $\lambda=-1+\sqrt{2}$. 
For the right graph,
the distances are  $a=-1$ and $b=-\lambda/(1+\lambda)$ where $\lambda$ is the second-smallest root of $x^3+x^2-2x-1$. 

\subsection{$(p,q)=(3,3)$, 14 graphs on 9 vertices}

\includegraphics[width=3cm]{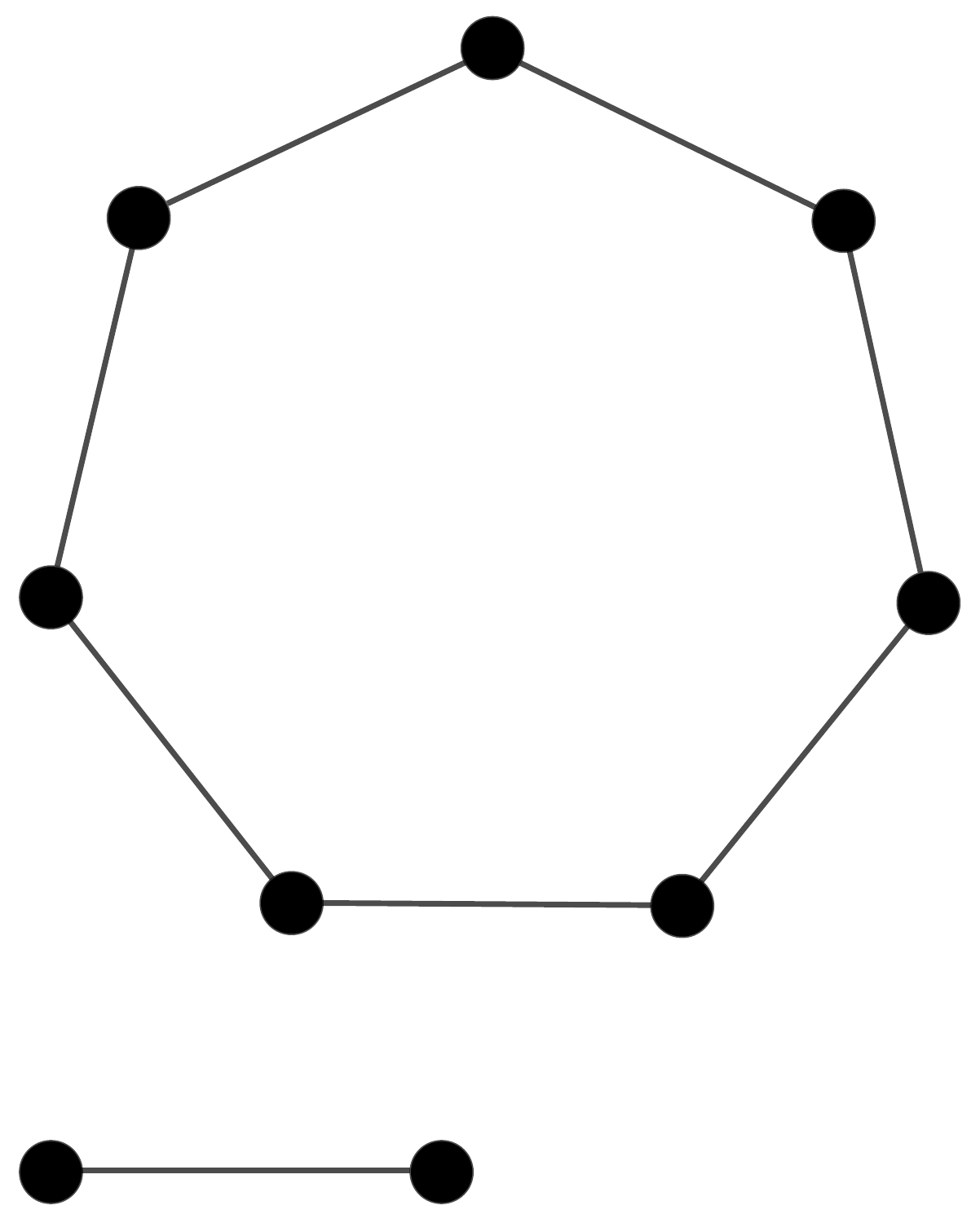}\hspace{1cm}
\includegraphics[width=3cm]{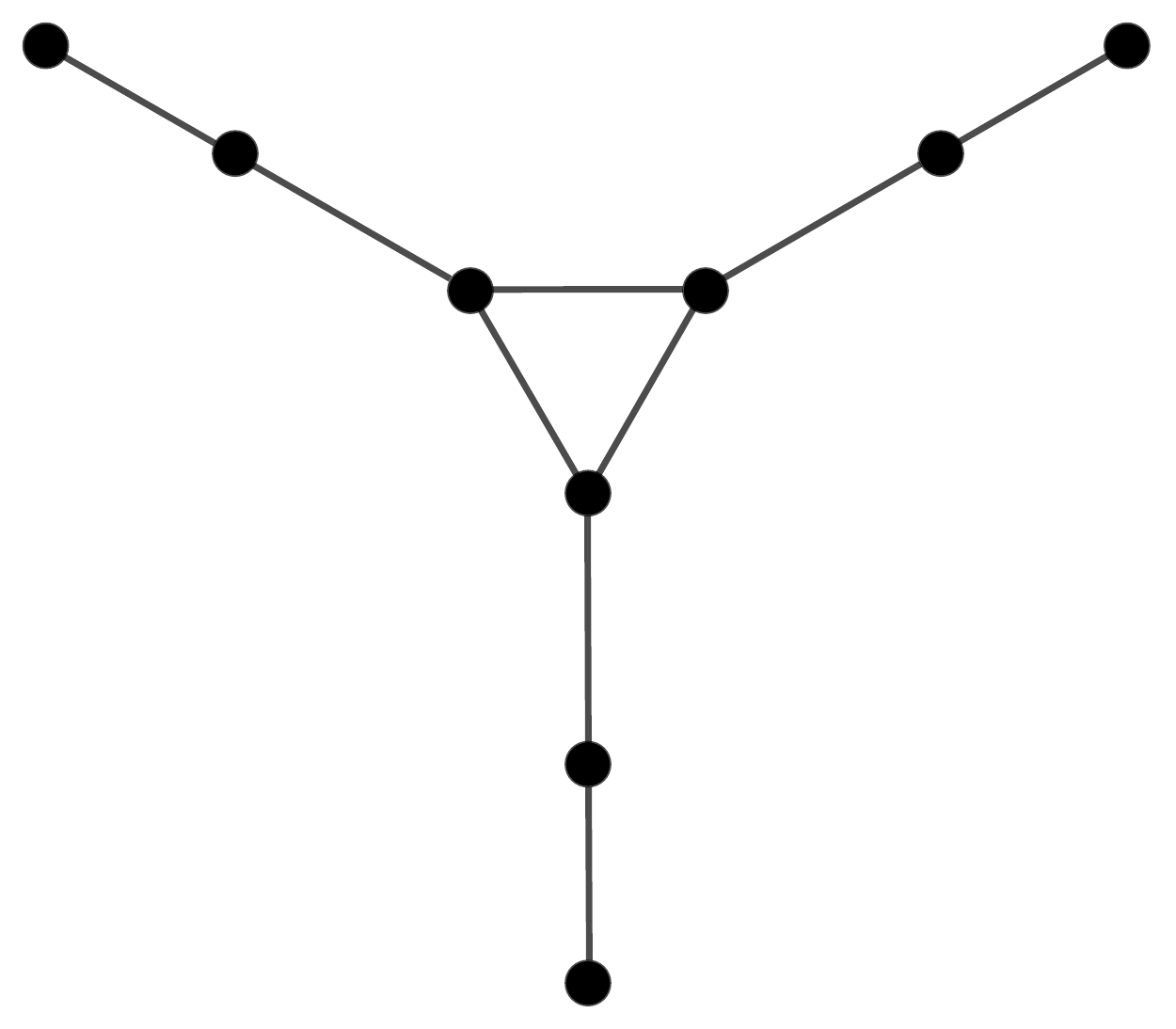}\hspace{1cm}
\includegraphics[width=3cm]{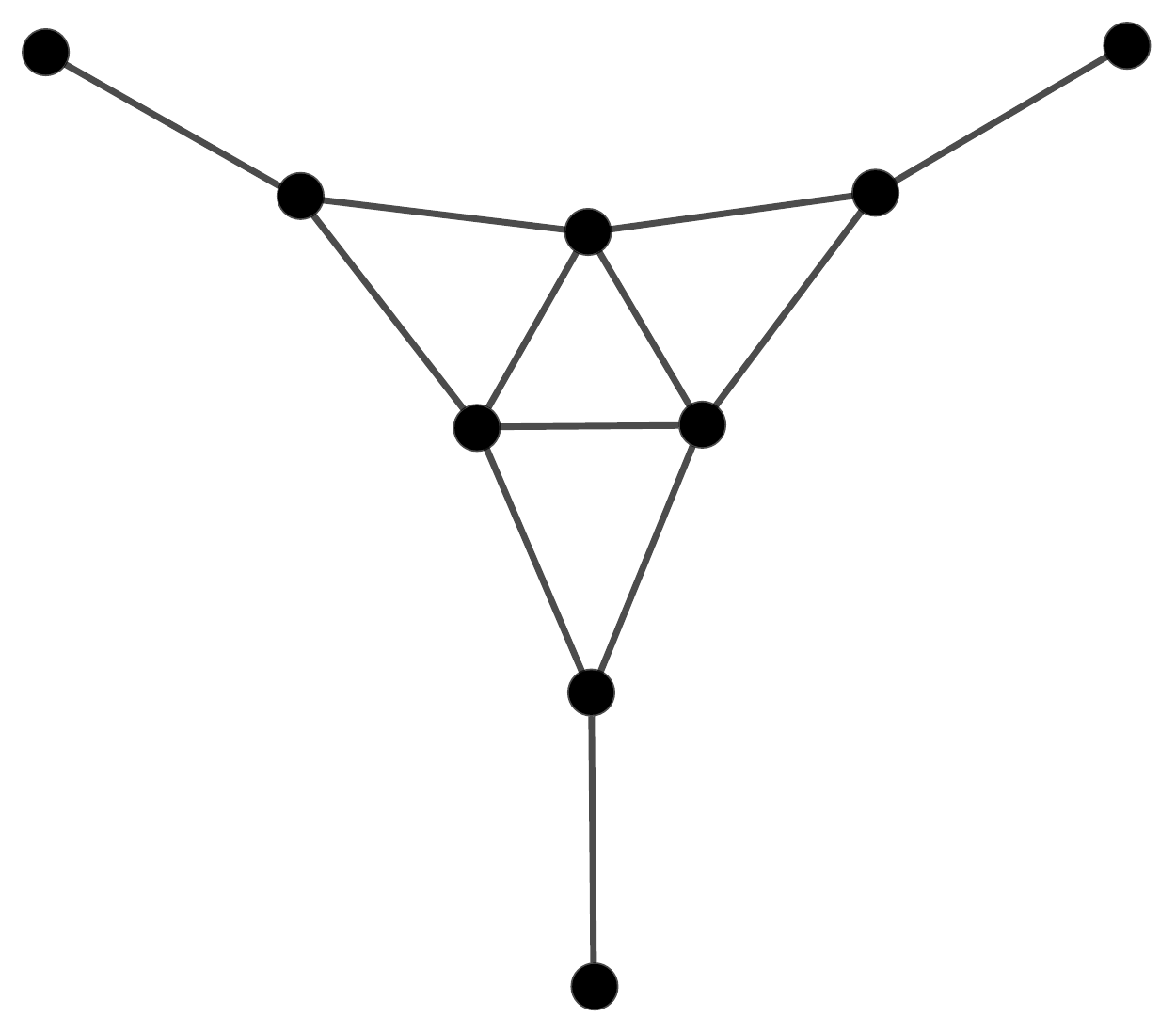}\hspace{1cm}
\includegraphics[width=3cm]{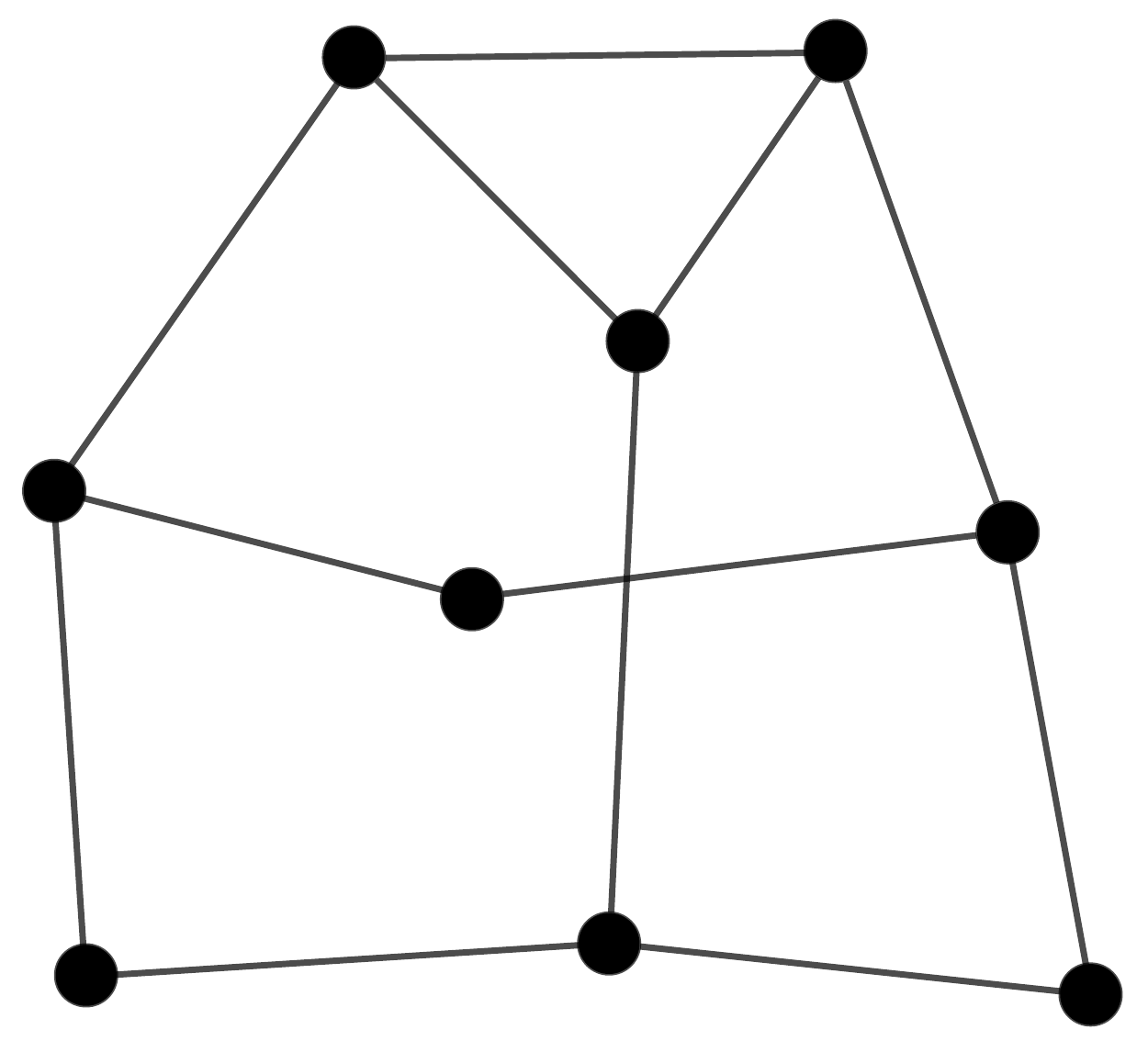}\hspace{1cm}
\\ \quad \\
\includegraphics[width=3cm]{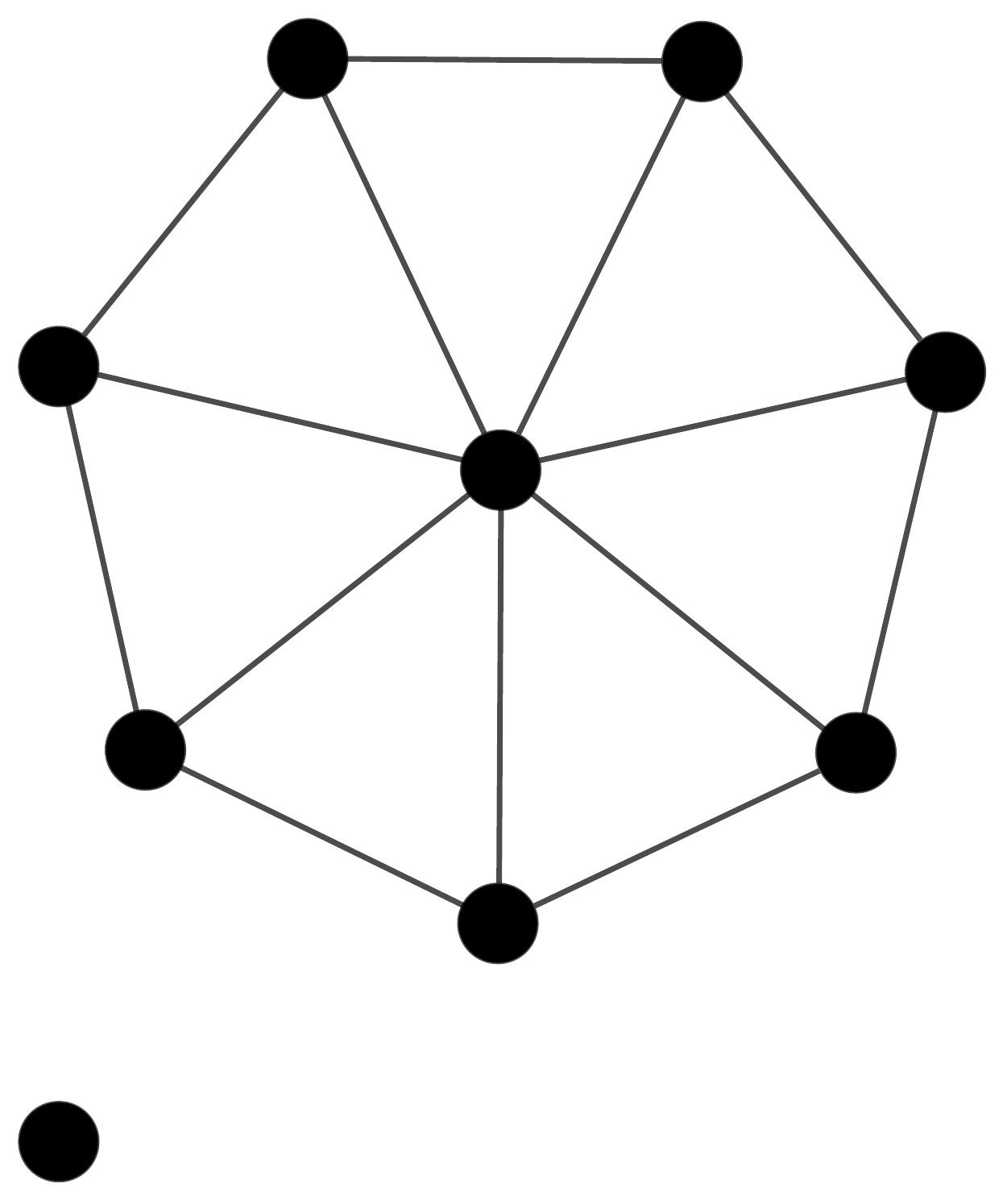}\hspace{1cm}
\includegraphics[width=3cm]{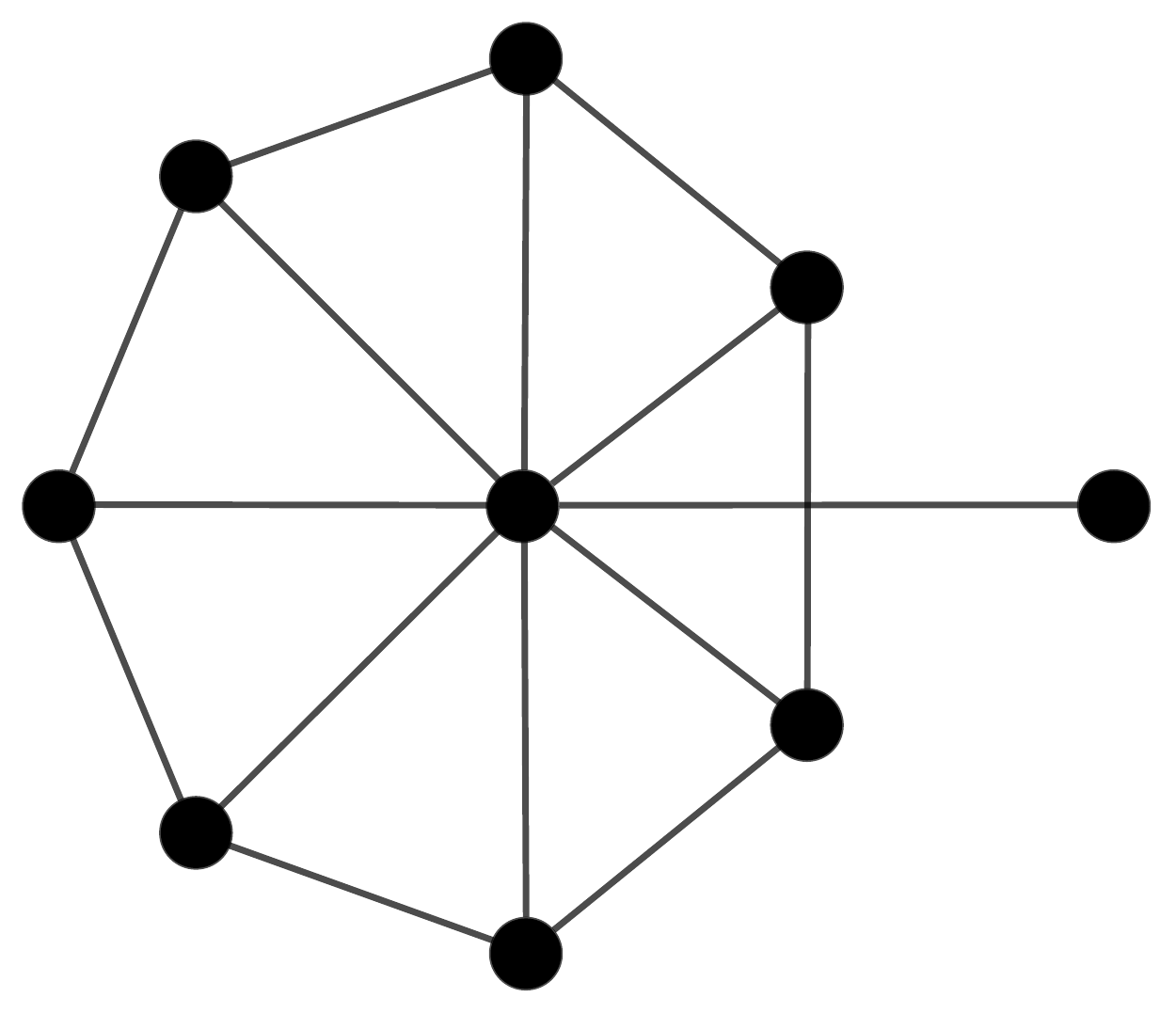}\hspace{1cm}
\includegraphics[width=3cm]{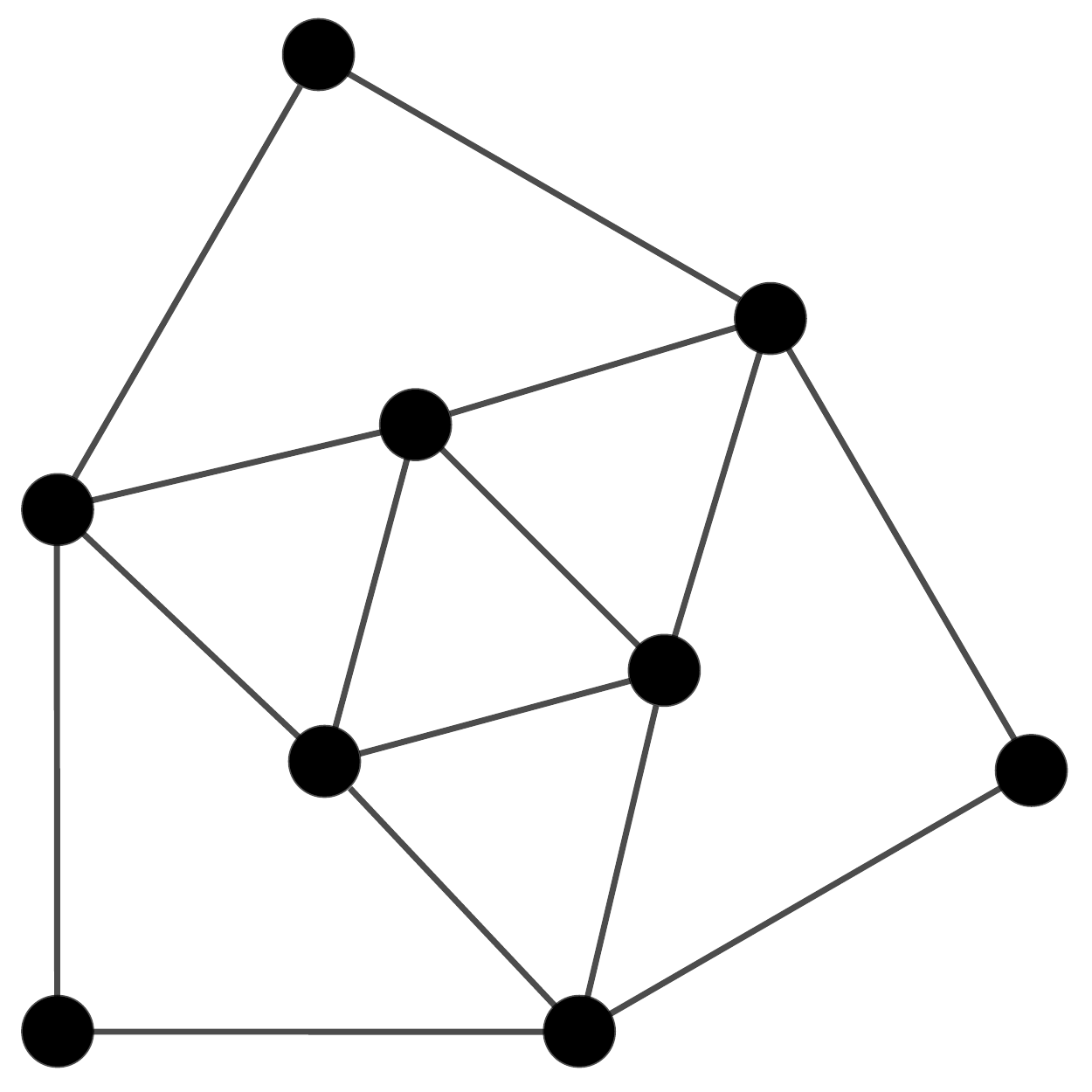}\hspace{1cm}
\includegraphics[width=3cm]{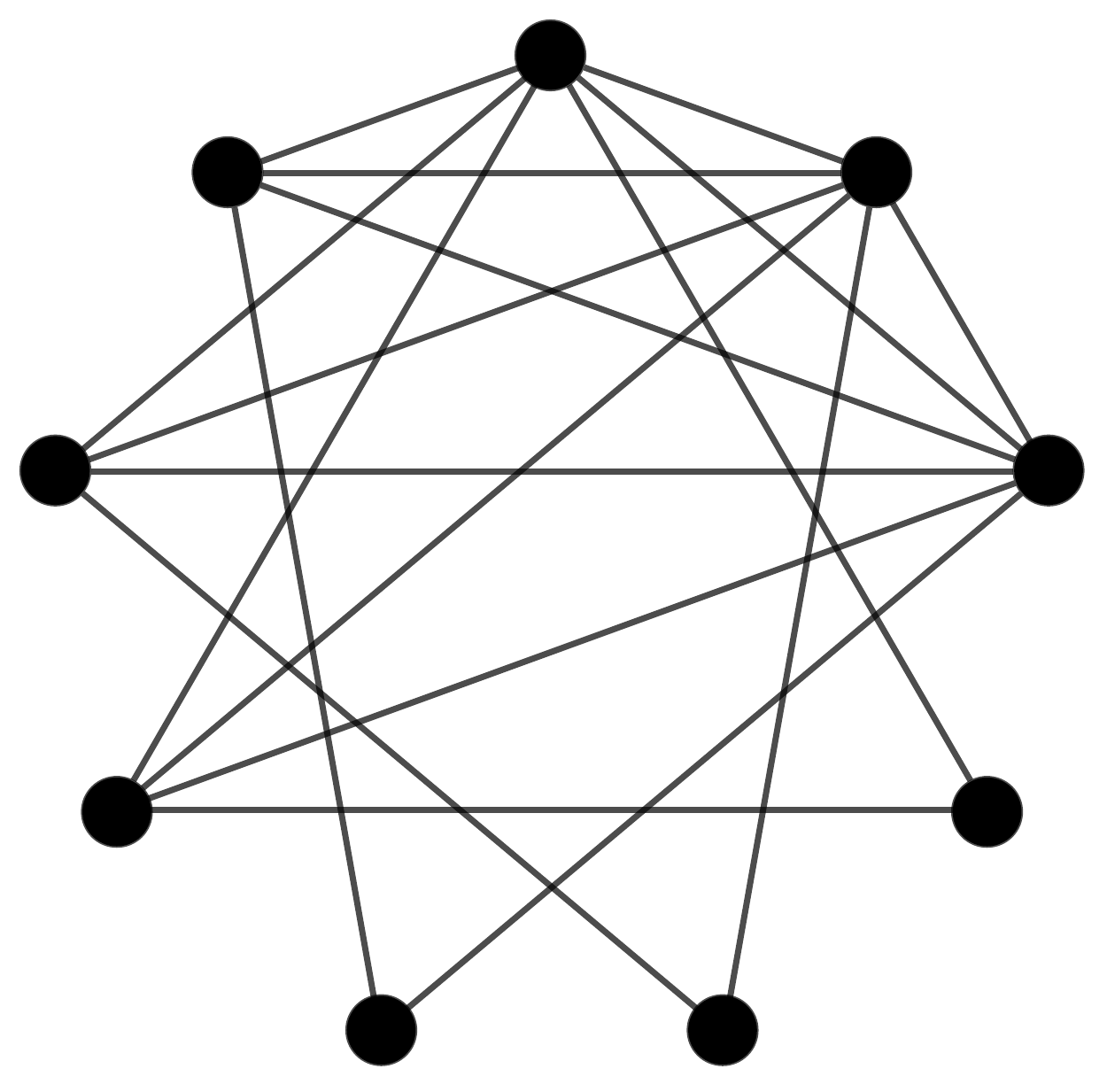}\hspace{1cm}
\\ \quad \\
\includegraphics[width=3cm]{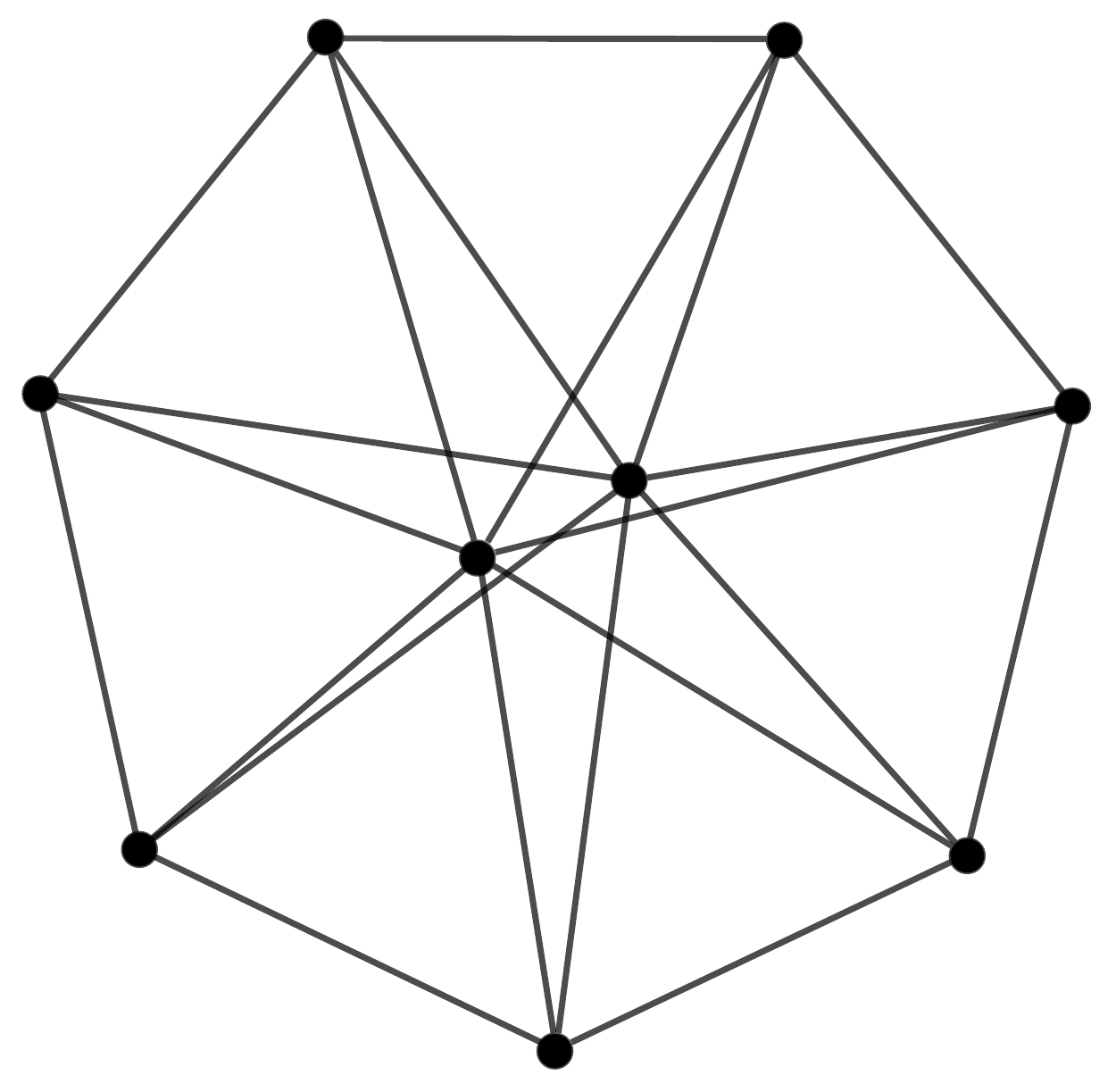}\hspace{1cm}
\includegraphics[width=3cm]{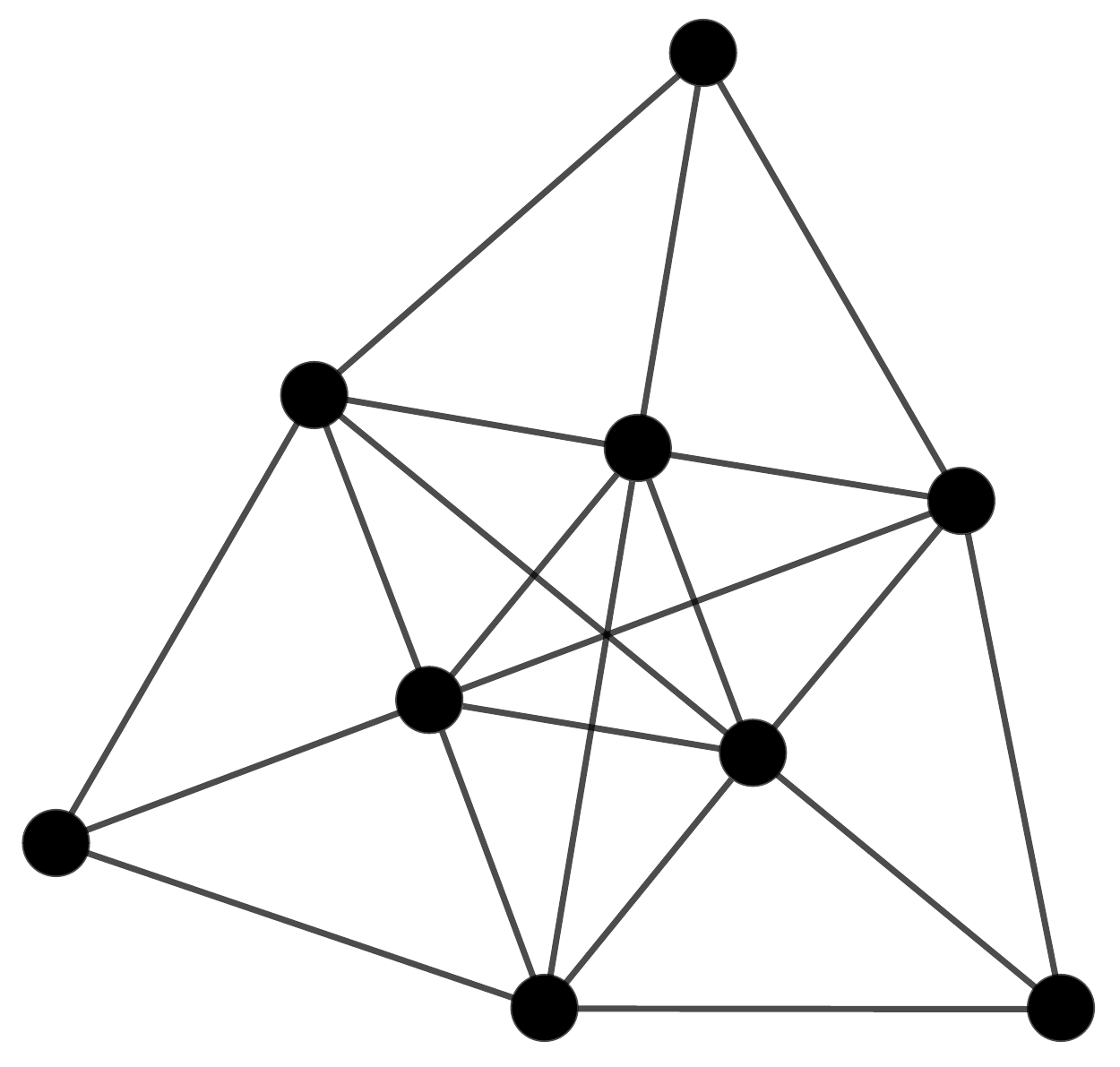}\hspace{1cm}
\\ \quad \\
For these 10 graphs, the distances are $a=1$ and $b=\lambda/(1+\lambda)$, where $\lambda$ is the second-smallest root of $x^3+x^2-2x-1$. 
\\ \quad \\
\includegraphics[width=3cm]{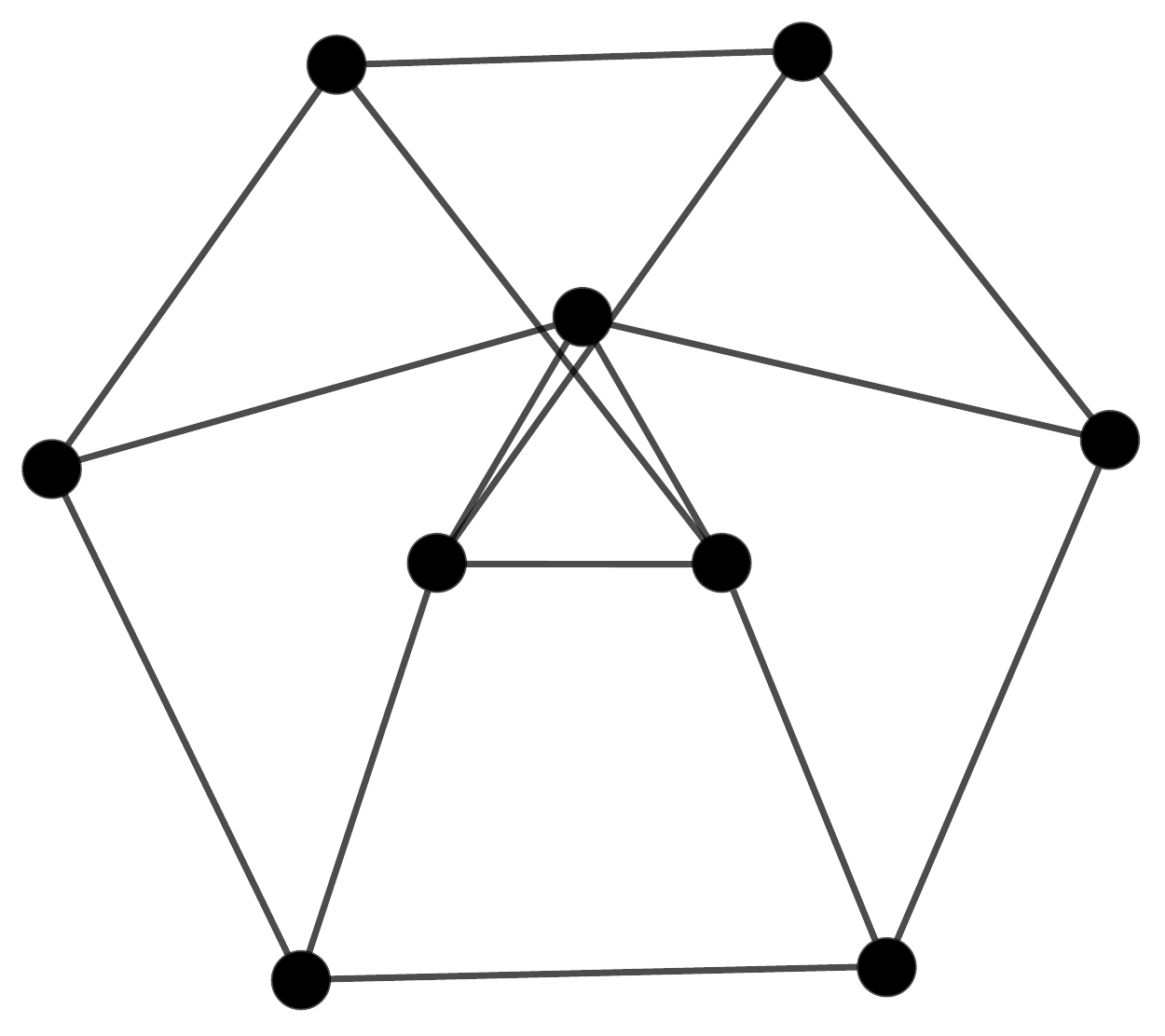}\hspace{1cm}
\includegraphics[width=3cm]{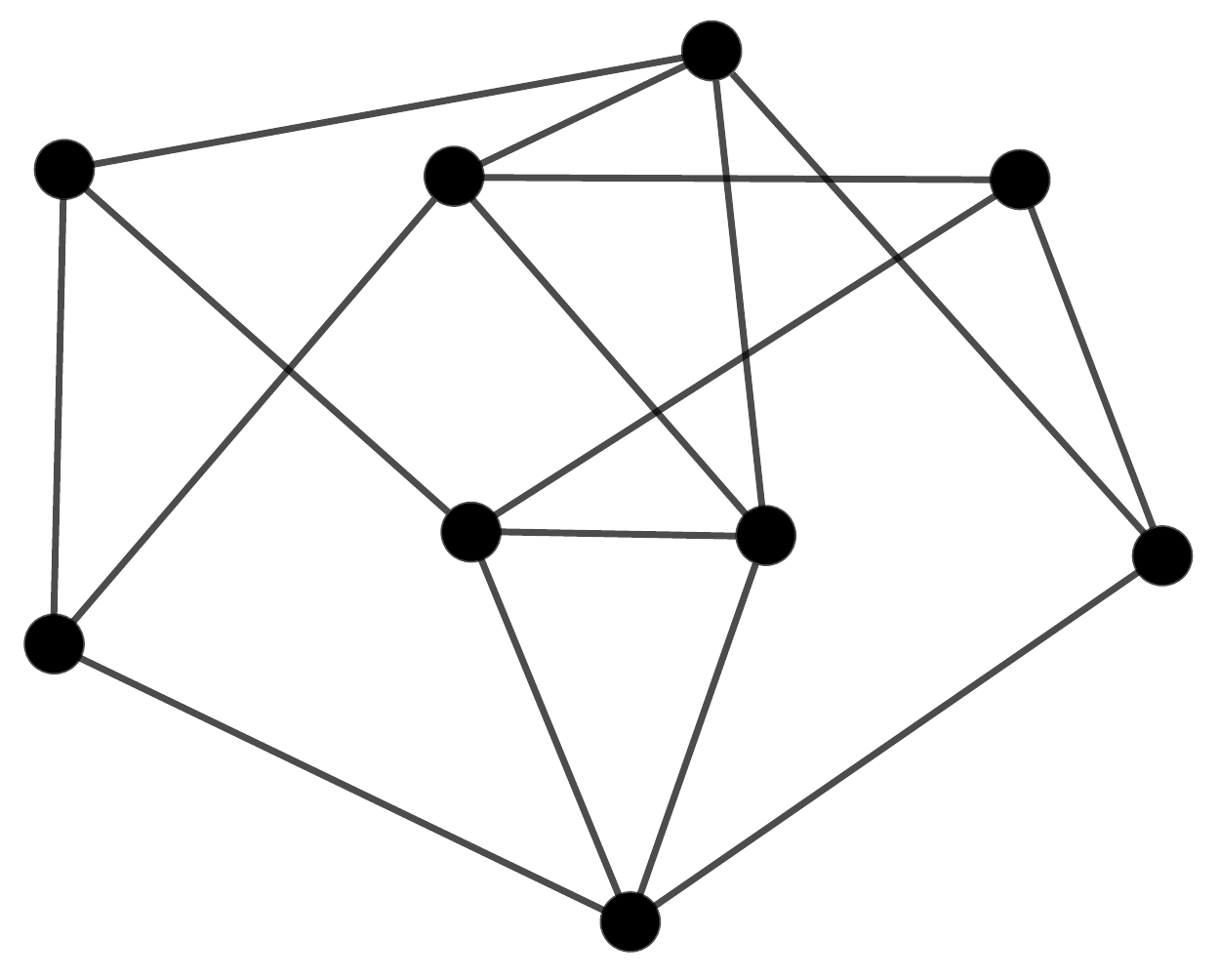}\hspace{1cm}
\\ \quad \\
For these 2 graphs, the distances are $a=1$ and $b=\lambda/(1+\lambda)$, where $\lambda=-1+\sqrt{2}$.  
\\ \quad \\
\includegraphics[width=3cm]{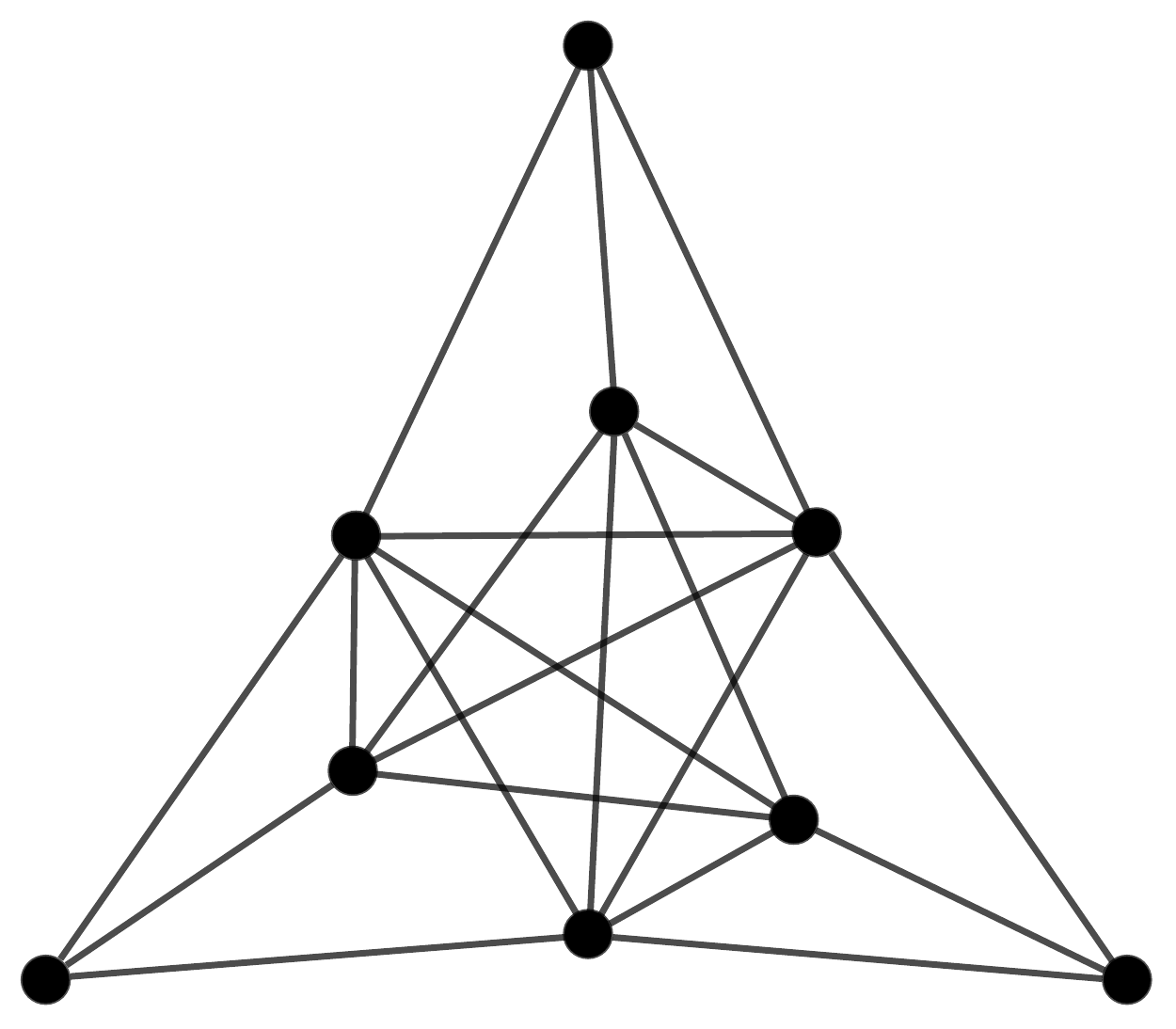}\hspace{1cm}
\includegraphics[width=3cm]{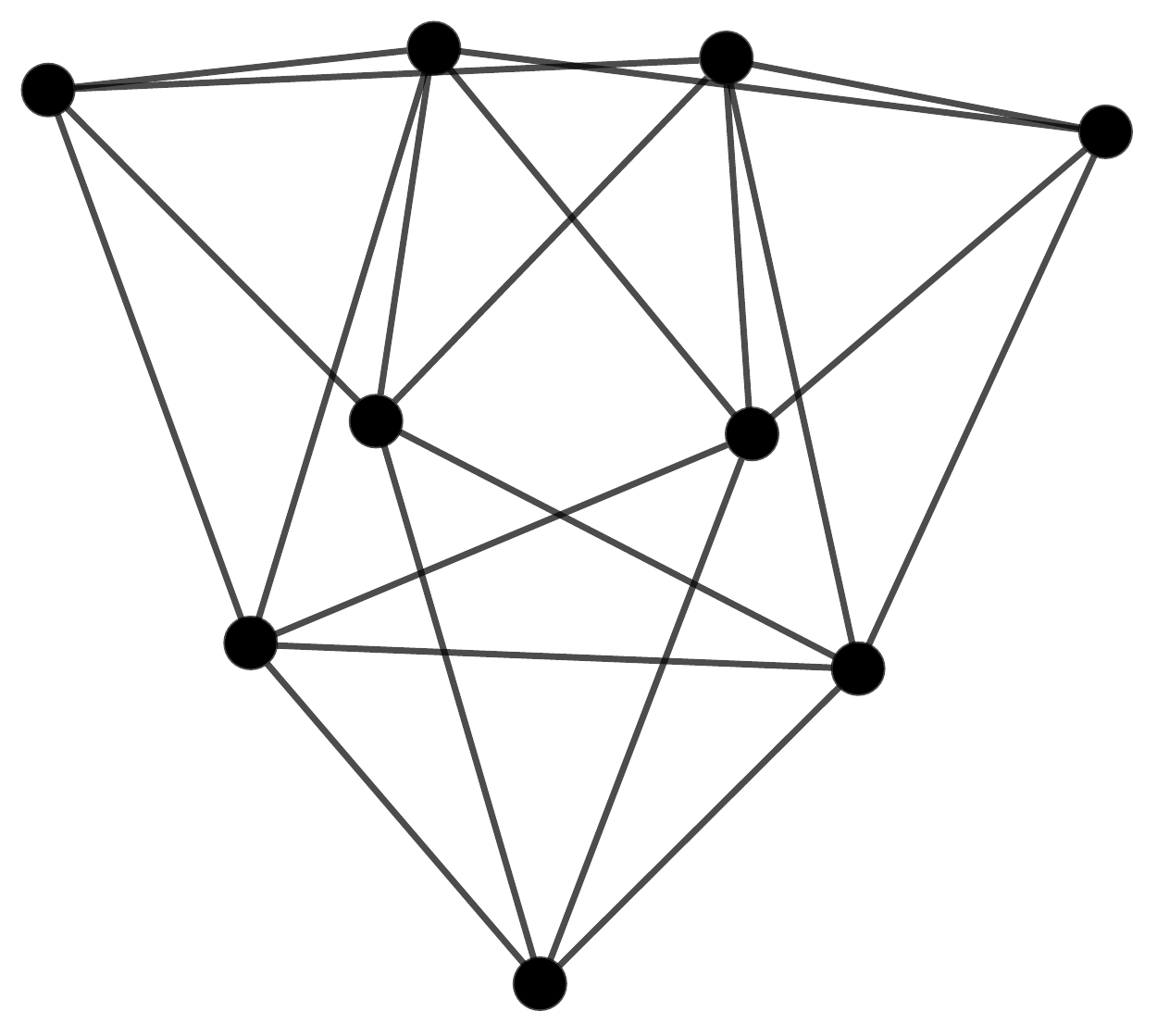}\hspace{1cm}
\\ \quad \\
For these 2 graphs, the distances are $a=1$ and $b=\lambda/(1+\lambda)$, where $\lambda=(-3+\sqrt{5})/2$. 

\subsection{$(p,q)=(4,1)$, 2 graphs on 10 vertices}
\includegraphics[width=3cm]{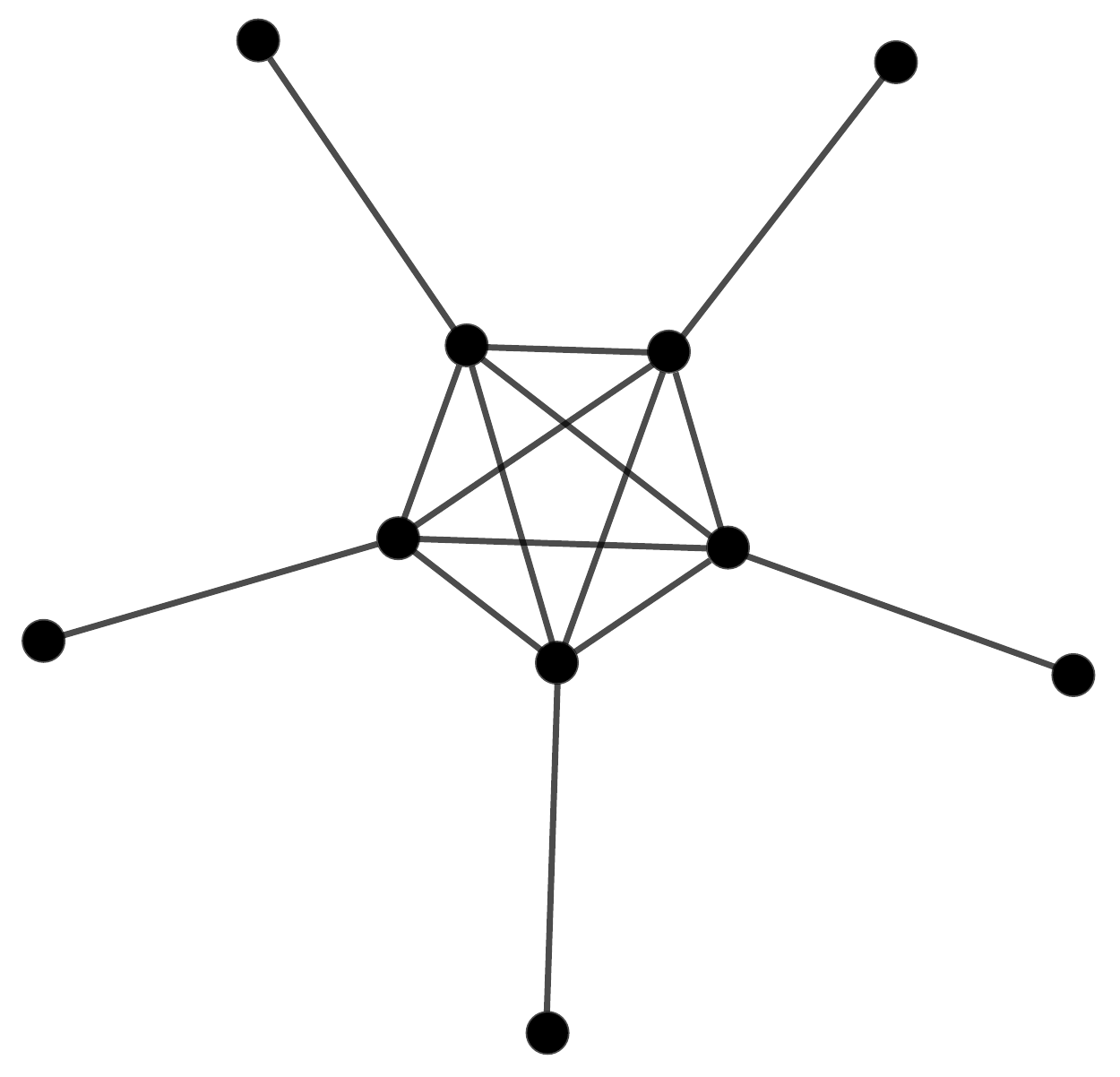}\hspace{1cm}
\includegraphics[width=1.5cm]{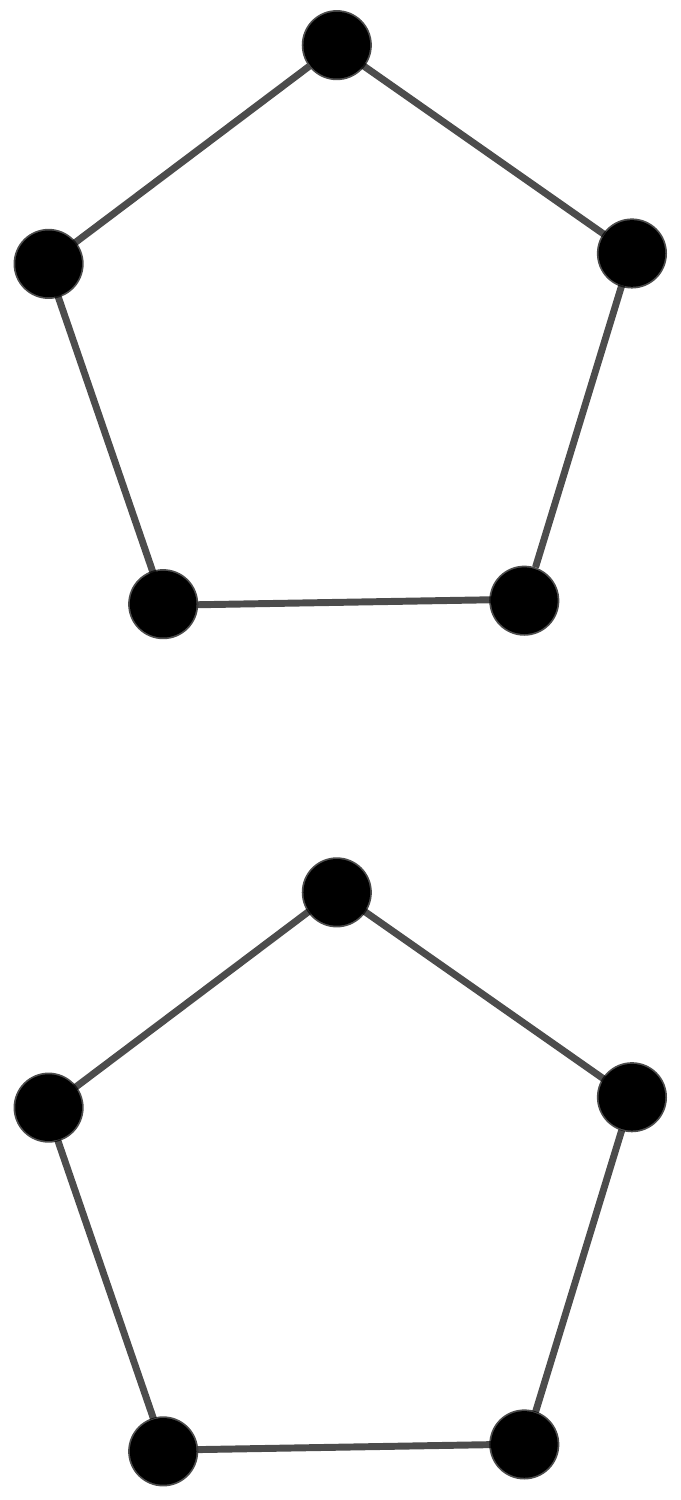}\hspace{1cm}
\\ \quad \\
For these 2 graphs, the distances are $a=1$ and $b=\lambda/(1+\lambda)$, where $\lambda=(-1+\sqrt{5})/2$. 

\subsection{$(p,q)=(4,2)$, 3 graphs on 10 vertices}
\includegraphics[width=4.5cm]{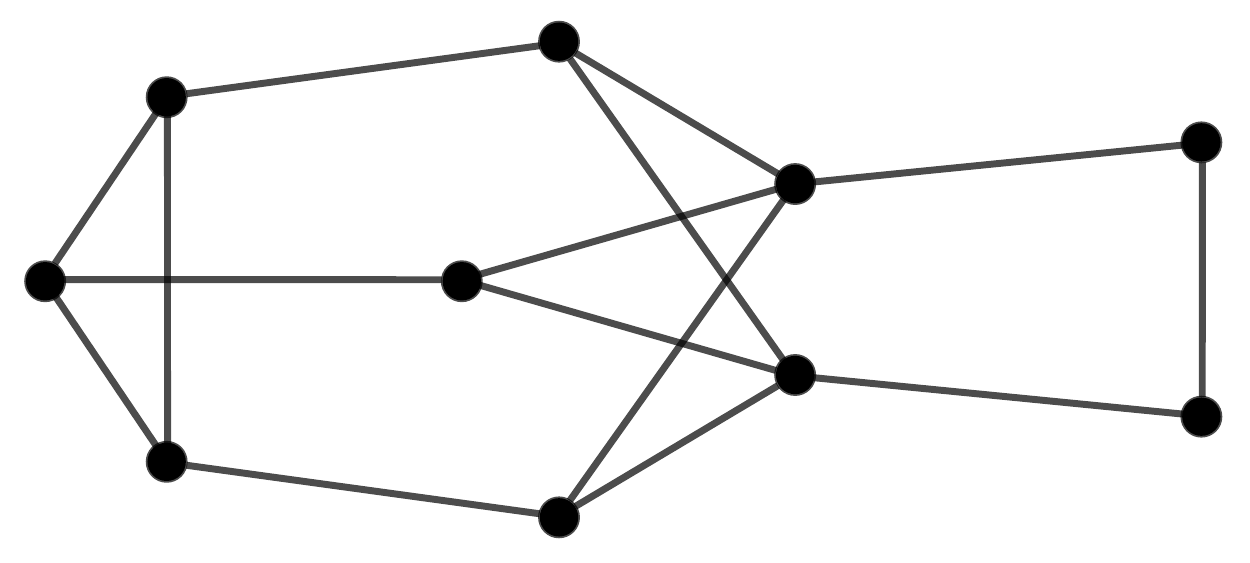}\hspace{1cm}
\\ \quad \\
For this graph, the distances are $a=1$ and $b=\lambda/(1+\lambda)$, where $\lambda=(-1+\sqrt{5})/2$. 
\\ \quad \\
\includegraphics[width=3cm]{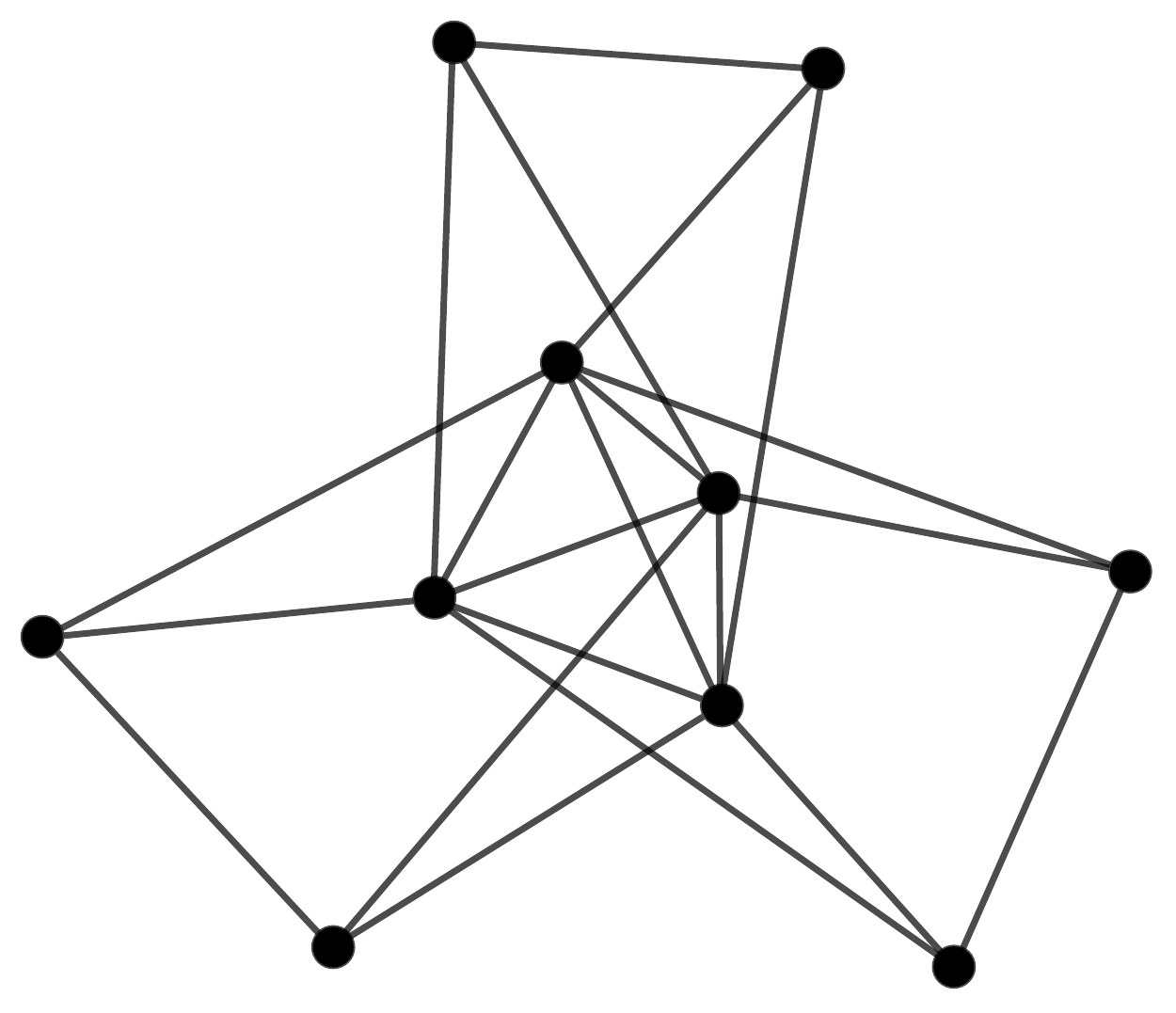}\hspace{1cm}
\includegraphics[width=3cm]{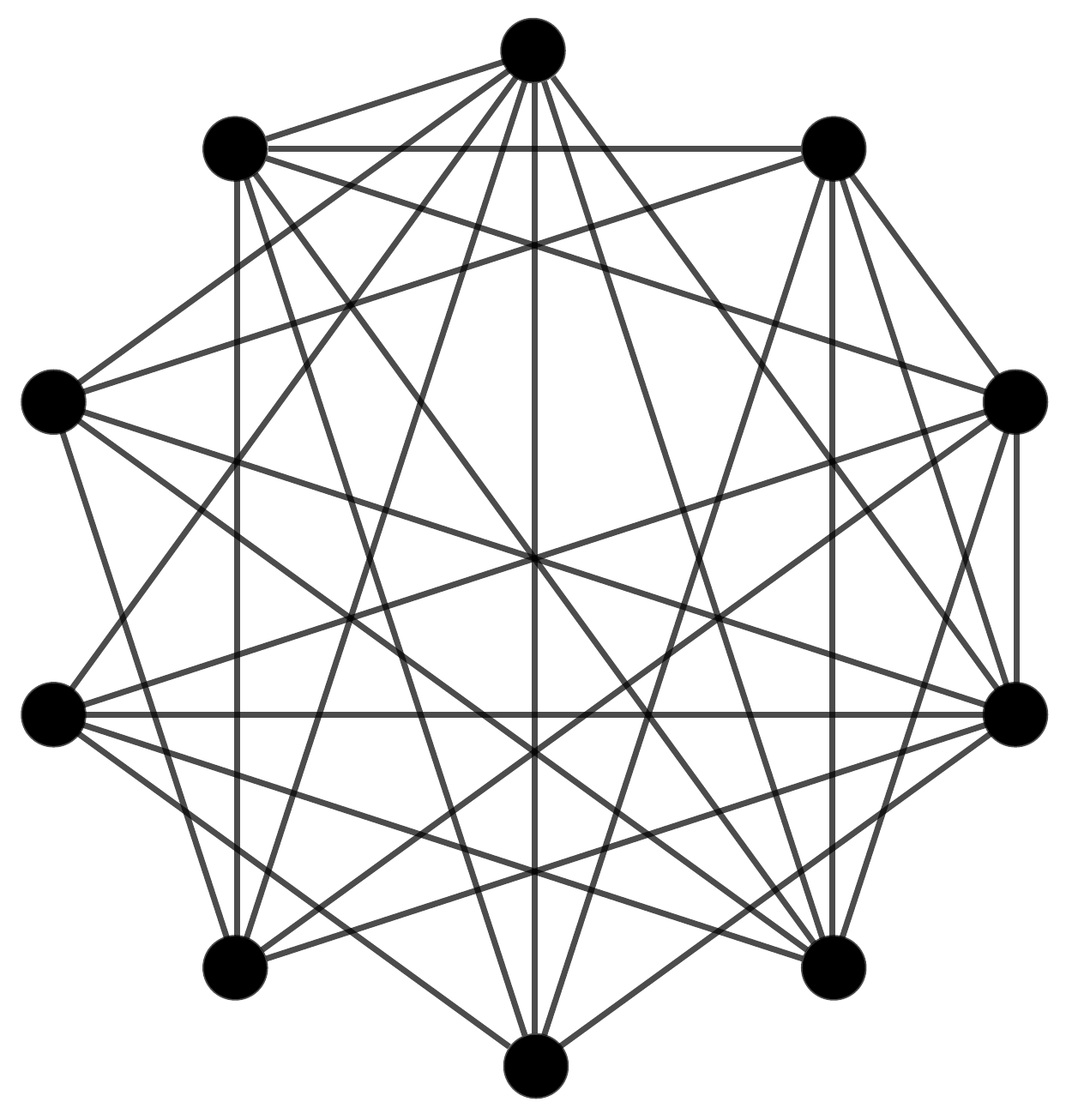}\hspace{1cm}
\\ \quad \\
For these 2 graphs, the distances are $a=1$ and $b=\lambda/(1+\lambda)$, where $\lambda=-1+\sqrt{2}$. 

\subsection{$(p,q)=(4,3)$, 1 graph on 12 vertices} 
\[
\left(
\begin{array}{cccccccccccc}
 0 & 0 & 0 & 0 & 0 & 1 & 1 & 1 & 0 & 1 & 1 & 1 \\
 0 & 0 & 0 & 0 & 0 & 1 & 0 & 0 & 1 & 1 & 1 & 0 \\
 0 & 0 & 0 & 0 & 0 & 0 & 1 & 1 & 1 & 0 & 1 & 0 \\
 0 & 0 & 0 & 0 & 0 & 0 & 0 & 1 & 1 & 1 & 0 & 1 \\
 0 & 0 & 0 & 0 & 0 & 0 & 0 & 0 & 0 & 0 & 0 & 0 \\
 1 & 1 & 0 & 0 & 0 & 0 & 0 & 1 & 1 & 0 & 0 & 0 \\
 1 & 0 & 1 & 0 & 0 & 0 & 0 & 0 & 1 & 1 & 0 & 0 \\
 1 & 0 & 1 & 1 & 0 & 1 & 0 & 0 & 0 & 1 & 1 & 0 \\
 0 & 1 & 1 & 1 & 0 & 1 & 1 & 0 & 0 & 0 & 0 & 1 \\
 1 & 1 & 0 & 1 & 0 & 0 & 1 & 1 & 0 & 0 & 1 & 0 \\
 1 & 1 & 1 & 0 & 0 & 0 & 0 & 1 & 0 & 1 & 0 & 1 \\
 1 & 0 & 0 & 1 & 0 & 0 & 0 & 0 & 1 & 0 & 1 & 0 \\
\end{array}
\right)
\]
For this graph, the distances are $a=1$ and $b=\lambda/(1+\lambda)$, where $\lambda=-1+\sqrt{2}$. 

\subsection{$(p,q)=(5,1)$, 3 graphs on 13 vertices}
\[ 
\left(
\begin{array}{ccccccccccccc}
 0 & 0 & 0 & 0 & 0 & 1 & 1 & 0 & 0 & 0 & 1 & 0 & 1 \\
 0 & 0 & 0 & 0 & 0 & 1 & 0 & 0 & 1 & 1 & 0 & 1 & 0 \\
 0 & 0 & 0 & 0 & 0 & 0 & 1 & 1 & 1 & 1 & 0 & 0 & 1 \\
 0 & 0 & 0 & 0 & 0 & 0 & 1 & 1 & 0 & 1 & 1 & 1 & 0 \\
 0 & 0 & 0 & 0 & 0 & 0 & 0 & 1 & 1 & 0 & 1 & 1 & 1 \\
 1 & 1 & 0 & 0 & 0 & 0 & 0 & 1 & 0 & 0 & 0 & 0 & 0 \\
 1 & 0 & 1 & 1 & 0 & 0 & 0 & 0 & 1 & 0 & 0 & 1 & 0 \\
 0 & 0 & 1 & 1 & 1 & 1 & 0 & 0 & 0 & 0 & 0 & 0 & 0 \\
 0 & 1 & 1 & 0 & 1 & 0 & 1 & 0 & 0 & 0 & 1 & 0 & 0 \\
 0 & 1 & 1 & 1 & 0 & 0 & 0 & 0 & 0 & 0 & 1 & 0 & 1 \\
 1 & 0 & 0 & 1 & 1 & 0 & 0 & 0 & 1 & 1 & 0 & 0 & 0 \\
 0 & 1 & 0 & 1 & 1 & 0 & 1 & 0 & 0 & 0 & 0 & 0 & 1 \\
 1 & 0 & 1 & 0 & 1 & 0 & 0 & 0 & 0 & 1 & 0 & 1 & 0 \\
\end{array} 
\right),
\left(
\begin{array}{ccccccccccccc}
 0 & 0 & 0 & 0 & 1 & 0 & 0 & 1 & 0 & 1 & 0 & 1 & 0 \\
 0 & 0 & 0 & 0 & 0 & 1 & 0 & 1 & 0 & 0 & 0 & 1 & 1 \\
 0 & 0 & 0 & 0 & 0 & 0 & 1 & 1 & 0 & 1 & 1 & 0 & 1 \\
 0 & 0 & 0 & 0 & 0 & 0 & 1 & 0 & 1 & 1 & 0 & 1 & 1 \\
 1 & 0 & 0 & 0 & 0 & 0 & 0 & 0 & 1 & 0 & 1 & 0 & 1 \\
 0 & 1 & 0 & 0 & 0 & 0 & 0 & 0 & 1 & 1 & 1 & 0 & 0 \\
 0 & 0 & 1 & 1 & 0 & 0 & 0 & 1 & 1 & 0 & 1 & 1 & 0 \\
 1 & 1 & 1 & 0 & 0 & 0 & 1 & 0 & 1 & 0 & 0 & 0 & 0 \\
 0 & 0 & 0 & 1 & 1 & 1 & 1 & 1 & 0 & 0 & 0 & 0 & 0 \\
 1 & 0 & 1 & 1 & 0 & 1 & 0 & 0 & 0 & 0 & 0 & 0 & 0 \\
 0 & 0 & 1 & 0 & 1 & 1 & 1 & 0 & 0 & 0 & 0 & 1 & 0 \\
 1 & 1 & 0 & 1 & 0 & 0 & 1 & 0 & 0 & 0 & 1 & 0 & 0 \\
 0 & 1 & 1 & 1 & 1 & 0 & 0 & 0 & 0 & 0 & 0 & 0 & 0 \\
\end{array}
\right),
\]
\[
\left(
\begin{array}{ccccccccccccc}
 0 & 0 & 0 & 0 & 0 & 1 & 1 & 0 & 0 & 0 & 0 & 1 & 1 \\
 0 & 0 & 0 & 0 & 0 & 1 & 0 & 0 & 0 & 1 & 1 & 0 & 0 \\
 0 & 0 & 0 & 0 & 0 & 0 & 1 & 1 & 0 & 0 & 1 & 0 & 0 \\
 0 & 0 & 0 & 0 & 0 & 0 & 0 & 1 & 1 & 1 & 0 & 1 & 0 \\
 0 & 0 & 0 & 0 & 0 & 0 & 0 & 0 & 1 & 0 & 1 & 0 & 1 \\
 1 & 1 & 0 & 0 & 0 & 0 & 0 & 1 & 1 & 0 & 0 & 0 & 0 \\
 1 & 0 & 1 & 0 & 0 & 0 & 0 & 0 & 1 & 1 & 0 & 0 & 0 \\
 0 & 0 & 1 & 1 & 0 & 1 & 0 & 0 & 0 & 0 & 0 & 0 & 1 \\
 0 & 0 & 0 & 1 & 1 & 1 & 1 & 0 & 0 & 0 & 0 & 0 & 0 \\
 0 & 1 & 0 & 1 & 0 & 0 & 1 & 0 & 0 & 0 & 0 & 0 & 1 \\
 0 & 1 & 1 & 0 & 1 & 0 & 0 & 0 & 0 & 0 & 0 & 1 & 0 \\
 1 & 0 & 0 & 1 & 0 & 0 & 0 & 0 & 0 & 0 & 1 & 0 & 0 \\
 1 & 0 & 0 & 0 & 1 & 0 & 0 & 1 & 0 & 1 & 0 & 0 & 0 \\
\end{array}
\right).
\]
For these 3 graphs, the distances $a=1$ and $b=1/2$. 

\subsection{$(p,q)=(5,2)$, 1 graph on 13 vertices}
\[ 
\left(
\begin{array}{ccccccccccccc}
 0 & 0 & 0 & 1 & 0 & 1 & 0 & 1 & 0 & 1 & 1 & 1 & 0 \\
 0 & 0 & 0 & 0 & 1 & 0 & 1 & 1 & 0 & 0 & 0 & 0 & 1 \\
 0 & 0 & 0 & 0 & 0 & 1 & 1 & 1 & 1 & 1 & 1 & 1 & 1 \\
 1 & 0 & 0 & 0 & 0 & 1 & 0 & 0 & 1 & 1 & 0 & 1 & 1 \\
 0 & 1 & 0 & 0 & 0 & 0 & 1 & 0 & 1 & 0 & 1 & 0 & 0 \\
 1 & 0 & 1 & 1 & 0 & 0 & 0 & 1 & 1 & 0 & 0 & 1 & 0 \\
 0 & 1 & 1 & 0 & 1 & 0 & 0 & 0 & 0 & 1 & 0 & 0 & 0 \\
 1 & 1 & 1 & 0 & 0 & 1 & 0 & 0 & 1 & 1 & 1 & 0 & 1 \\
 0 & 0 & 1 & 1 & 1 & 1 & 0 & 1 & 0 & 1 & 1 & 0 & 1 \\
 1 & 0 & 1 & 1 & 0 & 0 & 1 & 1 & 1 & 0 & 1 & 0 & 1 \\
 1 & 0 & 1 & 0 & 1 & 0 & 0 & 1 & 1 & 1 & 0 & 1 & 1 \\
 1 & 0 & 1 & 1 & 0 & 1 & 0 & 0 & 0 & 0 & 1 & 0 & 1 \\
 0 & 1 & 1 & 1 & 0 & 0 & 0 & 1 & 1 & 1 & 1 & 1 & 0 \\
\end{array}
\right).
\]
For this graph, the distances are $a=1$ and $b=\lambda/(1+\lambda)$, where $\lambda=-1+\sqrt{2}$. 
\subsection{$(p,q)=(6,1)$, 1 graph on 22 vertices} \label{sec:A12}
There exists the largest proper 2-indefinite-distance set in $\R^{6,1}$ with 22 points. 
The graph $G=(V,E)$ of this configuration is as follows. 
Let $V=V_0\cup V_1 \cup V_2$ be the vertex set, where
\begin{align*}
    V_0&=\{0 \}, \\
V_1&=\{1,2,3,4,5,6\}, \\
V_2&=\{\{i,j\} \mid 1\le i<j\le 6\}. 
\end{align*}
The edge set is $E=E_0 \cup E_1 \cup E_2$, where
\begin{align*}
    E_0&=\{\{ v,u \} \mid v \in V_0, u \in V_1 \}, \\
    E_1&=\{\{u,S\} \mid u \in V_1, S \in V_2, u \in S\},\\
    E_2&=\{\{S,T\} \mid S,T \in V_2, S\cap T =\emptyset \}.
\end{align*}
Note that the induced subgraph of $V_2$ is the complement of the Johnson graph, and the induced subgraph of $V_1$ are isolated vertices. 
The coordinates of a representation of this graph are 
\begin{align*}
    &v_0=\frac{2}{3}\sum_{i=1}^6  \ev_i, \\
    &v_{1,i}=-\ev_i+\ev_7\qquad  (0\leq i \leq 6),\\
    &v_{2,\{i,j\}}=\ev_i+\ev_j\qquad  (0\leq i<j\leq 6),
\end{align*}
where ${\ev}_i$ be the vector with 1 in $i$-th component and 0 in all others. 
The distances are $a=4$ and $b=2$.

\section{Figures or matrices of largest proper spherical 2-indefinite-distance sets} \label{sec:B}
For $(p,q)=(1,1)$, by applying Descartes’ rule of signs to $-\M_i$ in \eqref{eq:(1,1)}, we can prove that $\M_1$ is of Type (3), and $\M_2$ is of Type (2) and spherical. 
For $(p,q)=(2,1)$, all the largest proper 2-indefinite-distance sets of 5 points are Type (1). 
There are infinitely many spherical proper 2-indefinite-distance sets of 4 points in $\R^{2,1}$. For instance, the graph of order 4 with only 1 edge can be represented as a spherical proper 2-indefinite-distance set in $\R^{2,1}$ with $a=1$ and $0<b<1/4$. These are not the only examples of proper 2-indefinite-distance sets of order 4 in $S_{2,1}$. 
The largest 2-indefinite-distance set for $(p,q)=(4,3)$ is not spherical in the embedding dimension, and it is Type (1). 
The following graph
\begin{center}
    \includegraphics[width=3cm]{41no1.pdf}\hspace{1cm}
\end{center}
for $(p,q)=(4,1)$ is not spherical in the embedding dimension, and it is Type (3).  The remaining 2-indefinite-distance sets obtained in the previous section \ref{sec:A} are all spherical in the embedding dimension. Therefore we obtained Table~\ref{tab:2} on the cardinalities of the largest proper spherical 2-indefinite-distance sets. 
For $(p,q)=(4,3)$, there are 3 largest spherical proper 2-indefinite-distance sets, which are the following graphs. 
\[\left(
\begin{array}{ccccccccccc}
 0 & 0 & 0 & 0 & 1 & 0 & 0 & 1 & 0 & 0 & 1 \\
 0 & 0 & 0 & 0 & 0 & 1 & 0 & 1 & 0 & 1 & 0 \\
 0 & 0 & 0 & 0 & 0 & 0 & 1 & 1 & 1 & 0 & 0 \\
 0 & 0 & 0 & 0 & 0 & 0 & 0 & 1 & 1 & 1 & 1 \\
 1 & 0 & 0 & 0 & 0 & 0 & 0 & 0 & 1 & 1 & 0 \\
 0 & 1 & 0 & 0 & 0 & 0 & 0 & 0 & 1 & 0 & 1 \\
 0 & 0 & 1 & 0 & 0 & 0 & 0 & 0 & 0 & 1 & 1 \\
 1 & 1 & 1 & 1 & 0 & 0 & 0 & 0 & 1 & 1 & 1 \\
 0 & 0 & 1 & 1 & 1 & 1 & 0 & 1 & 0 & 1 & 1 \\
 0 & 1 & 0 & 1 & 1 & 0 & 1 & 1 & 1 & 0 & 1 \\
 1 & 0 & 0 & 1 & 0 & 1 & 1 & 1 & 1 & 1 & 0 \\
\end{array}
\right),
\left(
\begin{array}{ccccccccccc}
 0 & 0 & 0 & 0 & 1 & 1 & 1 & 1 & 0 & 1 & 1 \\
 0 & 0 & 0 & 0 & 1 & 0 & 1 & 0 & 1 & 1 & 0 \\
 0 & 0 & 0 & 0 & 0 & 1 & 1 & 1 & 1 & 0 & 0 \\
 0 & 0 & 0 & 0 & 0 & 0 & 0 & 1 & 1 & 1 & 1 \\
 1 & 1 & 0 & 0 & 0 & 0 & 0 & 1 & 1 & 0 & 0 \\
 1 & 0 & 1 & 0 & 0 & 0 & 0 & 0 & 1 & 1 & 0 \\
 1 & 1 & 1 & 0 & 0 & 0 & 0 & 1 & 0 & 1 & 1 \\
 1 & 0 & 1 & 1 & 1 & 0 & 1 & 0 & 0 & 1 & 0 \\
 0 & 1 & 1 & 1 & 1 & 1 & 0 & 0 & 0 & 0 & 1 \\
 1 & 1 & 0 & 1 & 0 & 1 & 1 & 1 & 0 & 0 & 0 \\
 1 & 0 & 0 & 1 & 0 & 0 & 1 & 0 & 1 & 0 & 0 \\
\end{array}
\right). \]
For these 2 graphs, the distances are $a=1$ and $b=\lambda/(1+\lambda)$, where $\lambda=-1+\sqrt{2}$. 
\[
\left(
\begin{array}{ccccccccccc}
 0 & 0 & 0 & 1 & 1 & 1 & 1 & 0 & 1 & 0 & 1 \\
 0 & 0 & 0 & 1 & 0 & 1 & 1 & 1 & 1 & 1 & 0 \\
 0 & 0 & 0 & 0 & 1 & 1 & 0 & 1 & 1 & 1 & 1 \\
 1 & 1 & 0 & 0 & 0 & 0 & 1 & 0 & 1 & 1 & 1 \\
 1 & 0 & 1 & 0 & 0 & 0 & 1 & 1 & 1 & 1 & 0 \\
 1 & 1 & 1 & 0 & 0 & 0 & 0 & 1 & 0 & 1 & 1 \\
 1 & 1 & 0 & 1 & 1 & 0 & 0 & 0 & 1 & 0 & 1 \\
 0 & 1 & 1 & 0 & 1 & 1 & 0 & 0 & 0 & 1 & 1 \\
 1 & 1 & 1 & 1 & 1 & 0 & 1 & 0 & 0 & 0 & 0 \\
 0 & 1 & 1 & 1 & 1 & 1 & 0 & 1 & 0 & 0 & 0 \\
 1 & 0 & 1 & 1 & 0 & 1 & 1 & 1 & 0 & 0 & 0 \\
\end{array}
\right).\]
For this graph the distances are $a=1$ and $b=\lambda/(1+\lambda)$, where $\lambda=(-3+\sqrt{5})/2$. 

\end{document}